\documentclass[12pt]{article}
\usepackage{amsmath}
\usepackage{amsfonts}
\usepackage{amssymb}
\usepackage{mathrsfs}
\usepackage{graphicx}
\usepackage[all]{xy}
\usepackage{theorem}
\usepackage{authblk}

\newcommand{\batuint}[2]{\ooalign{$\displaystyle\int_{#1}{#2}$ \crcr \hspace{0.7pt}$\times$}}
\newcommand{\ywrk}[1]{\mathcal{#1}}
\newcommand{\defeq}{:=}
\newcommand{\meas}[2]{{\rm Meas}\lt(#1, #2\rt)}
\newcommand{\ratint}{\mathbb Z}
\newcommand{\homm}{{\rm Hom}}
\newcommand{\padhfpl}[1]{\mathcal{H}_\Gamma \hspace{-2pt}\left(#1\right)}
\newcommand{\mws}[1]{\rotatebox{90}{#1}}
\newcommand{\ratnum}{\mathbb{Q}}
\newcommand{\adele}{\mathbb{A}}
\newcommand{\pee}{\mathfrak{p}}

\newcommand{\intring}[1]{\mathcal{O}_{#1}}
\newcommand{\cpxnum}{\mathbb{C}}
\newcommand{\rmf}[1]{{\rm #1}}
\newcommand{\gyoretu}[4]{\begin{pmatrix} #1 & #2 \\ #3 & #4 \end{pmatrix}}
\newcommand{\relmiddle}[1]{\mathrel{}\middle#1\mathrel{}}
\newcommand{\doi}[1]{\mathfrak{#1}}
\newcommand{\nnm}[1]{{\it #1}}
\newcommand{\bfb}[1]{{\mathbf #1}}
\newcommand{\bfw}[1]{{\mathbb #1}}
\newcommand{\dlangle}{\langle\hspace{-0mm}\left\langle}
\newcommand{\drangle}{\rangle\hspace{-0mm}\right\rangle}
\newcommand{\norm}[2]{N_{#1}\hspace{-0pt}(#2)}
\newcommand{\uesen}[1]{\overline{#1}}
\newcommand{\fxdring}{\mathcal{O}}
\newcommand{\nami}[1]{\widetilde{#1}}
\newcommand{\yama}[1]{\widehat{#1}}
\newcommand{\lwari}[2]{\lower0pt\hbox{#1} \big{\backslash}  \raise0pt\hbox{#2}}
\newcommand{\dblcoset}[3]{\lower0pt\hbox{#1} \big{\backslash} \raise0pt\hbox{#2} \hspace{-2pt}\big{/}\lower0pt\hbox{#3}}
\newcommand{\rcoset}[2]{\raise0pt\hbox{#1} \big{/}\lower0pt\hbox{#2}}
\newcommand{\posuni}{\doi{c}}
\newcommand{\ds}{\displaystyle}
\newcommand{\mapdiag}[5]{\xymatrix{{#1}\hspace{-12mm}&\ #2 \ar[r] & #3 \\ & #4 \ar@{}[u]|{\mws{$\in$}} \ar@{|->}[r] & #5 \ar@{}[u]|{\mws{$\in$}}}}
\newcommand{\lt}{\left}
\newcommand{\rt}{\right}
\newcommand{\bbs}{\big{\backslash}}
\newcommand{\spp}{\doi{p}}

\newcommand{\logNorm}{\rmf{logNorm}}
\newcommand{\swtcha}[2]{\langle\chi_{F,\rmf{cycl}}{#1}\rangle^{#2}}

\newcommand{\prjl}[1]{\bfw{P}^1(#1)}
\newcommand{\C}{\bfw{C}_p}
\newcommand{\dirlim}[1]{\lim_{\substack{\longrightarrow \\ #1}}}
\newcommand{\hfpl}[1]{\ywrk{H}\lt(#1\rt)}
\newcommand{\red}{\rmf{red}}
\newcommand{\ord}{\rmf{ord}}
\newcommand{\thetab}{\uesen{\theta}}
\newcommand{\Ib}{\uesen{I}}

\newcommand{\inv}{\rmf{inv}}
\def\qed{\hfill $\Box$}

\newtheorem{thm}{Theorem}[section]
\newtheorem{lem}[thm]{Lemma}
\newtheorem{defn}[thm]{Definition}
\newtheorem{rmk}[thm]{Remark}
\newtheorem{prop}[thm]{Proposition}
\newtheorem{ass}[thm]{Assumption}
\theorembodyfont{\upshape}
\newtheorem{proof}{Proof}

\newtheorem{notation}[thm]{Notation}

\title{Integrals on $p$-adic upper half planes and Hida families over totally real fields}
\author{Isao Ishikawa}
\begin{document}
\maketitle

\begin{abstract}
Bertolini-Darmon and Mok proved a formula of the second derivative of the two-variable $p$-adic $L$-function of a modular elliptic curve over a totally real field along the Hida family in terms of the image of a global point by some $p$-adic logarithm map. The theory of $p$-adic indefinite integrals and $p$-adic multiplicative integrals on $p$-adic upper half planes plays an important role in their work.  In this paper, we generalize these integrals for $p$-adic measures which are not necessarily $\ratint$-valued, and prove a formula of the second derivative of the two-variable $p$-adic $L$-function of an abelian variety of $\rmf{GL}(2)$-type associated to a Hilbert modular form of weight 2.
\end{abstract}

\tableofcontents

\section{Introduction}

Bertolini-Darmon and Mok proved a formula of the second derivative of the two-variable $p$-adic $L$-function of a modular elliptic curve over a totally real field along the Hida family in terms of the image of a global point by some $p$-adic logarithm map (\cite{BD1}, \cite{Mo1}). In this paper, we generalize their results to abelian varieties of $\rmf{GL}(2)$-type associated to Hilbert modular forms of weight 2.

Let $F$ be a totally real field. We assume that an odd prime number $p$ is inert in $F$ and denote by $\doi{p}$ the prime of $F$ above $p$. 
Let $f$ be a cuspidal Hilbert modular eigenform of parallel weight 2 over $F$. We assume that $f$ is a newform of level $\Gamma_0(\doi{n})$ (here, $\doi{n}$ is a non-zero ideal of $\intring{F}$) and the sign $\epsilon_f$ of the functional equation of the complex $L$-function of $f$ is equal to $-1$. Let $\ratnum(f)$ be the Hecke fields of $f$, which is a finite extension of $\ratnum$ generated by the Fourier coefficients of $f$.
 
Let $A$ be an abelian variety of $\rmf{GL}(2)$-type over $F$ associated to $f$. We assume that A has split multiplicative reduction at $\doi{p}$ (in addition, if $[F:\ratnum]$ is odd, suppose that $A$ is multiplicative at some prime other than $\doi{p}$).  We denote by $L_p(s,k)$ the {\em two variable $p$-adic $L$-function along the Hida family of $f$}.

\begin{thm}[see Theorem \ref{main results 3}]

\begin{eqnarray}\lt.\frac{d^2}{dk^2}L_p(k/2,k)\rt|_{k=2}=l\cdot\lt(\logNorm_p^{A}(P)\rt)^2\label{mainresultintro}\end{eqnarray}
where $l\in \ratnum(f)^\times$ and $P \in A(F)\otimes_\ratint\ratnum$ is a global point. The map \[\logNorm_p^A: A(\cpxnum_p)\otimes_\ratint\ratnum\longrightarrow \cpxnum_p\] is a $p$-adic logarithm map (see Definition \ref{logNormnodefdayon}).
\end{thm}
Here, the Hecke character $\psi$ in Theorem \ref{main results 3} is the trivial character.

Let us explain the outline of the proof.  Let $\Phi$ be an automorphic form on the multiplicative group of a definite quaternion algebra $B/F$ corresponding to $f$ by the Jacquet-Langlands correspondence (we can find such a quaternion algebra $B$ by the assumption of the reduction of the abelian variety $A$) .  The key notion for proving the formula (\ref{mainresultintro}) is the notion of {\em indefinite integrals} and {\em multiplicative integrals} associated to $\Phi$. In fact, we prove the following equalities:
\[\text{LHS of (\ref{mainresultintro})}=(\text{indefinite integral})=(\text{multiplicative integral})=\text{RHS of (\ref{mainresultintro})}\]

The first equality is proved by using an explicit formula of $L$-values by Gross-Hatcher and Xue. In the work of Bertolini-Darmon and Mok, multiplicative integrals are defined only when $A$ is an elliptic curve. We shall modify the definition of multiplicative integrals by following Dasgupta's method, and prove the second and  third equalities. For the third equality, we use the theory of $p$-adic uniformization of Shimura curves by Manin-Drinfeld and Cerednik-Drinfeld (see Section \ref{L-invariants} and  \ref{p-adic uniformization of Shimura curves}). For the second equality, we shall prove the following generalized formula of $p$-adic integrals on $p$-adic upper half planes $\ywrk{H}$:

\begin{thm}[see Theorem \ref{indefinite integral no thm}]
\label{intronothm2}
Let $\tau_1,\tau_2 \in \hfpl{\uesen{\ratnum}_p}$ be $\uesen{\ratnum}_p$-valued points on $p$-adic upper half plane. We have:
\[I_\Phi(\tau_1)-I_\Phi(\tau_2) = \iota\lt\{\lt(\lt(\logNorm_p+2\alpha_{\spp}\cdot\alpha_{\spp}'(0)\cdot\ord_\spp\rt)\otimes_\ratint\rmf{id}_{\intring{\ratnum(\Phi)}}\rt)\lt(\batuint{[\tau_1]-[\tau_2]}{\omega_{\mu_\Phi}}\rt)\rt\},\]
where $I_\Phi$ is the indefinite integral (see Section \ref{indefinite integral}) and $\alpha_\spp$ (resp.\, an $p$-adic analytic function $\alpha_\spp(s)$) is the Hecke eigenvalue of $T(\spp)$ of $\Phi$ (resp.\, of the Hida family associated to $\Phi$). The symbol $\batuint{}{}$ is the multiplicative integral whose integrated values are in $\cpxnum_p\otimes_\ratint\intring{\ratnum(\Phi)}$ (see Section \ref{batusekibun}). The map $\iota$ is a natural multiplication map $\cpxnum_p\otimes_\ratint\intring{\ratnum(\Phi)}\rightarrow\cpxnum_p$ (we denote by $\ratnum(\Phi)$ the Hecke field of $\Phi$). 
\end{thm}
When $A$ is an elliptic curve, this formula is proved by Bertolini-Darmon and Mok (\cite{BD1}, \cite{Mo1}).  For the proof of Theorem \ref{intronothm2}, see Section \ref{indefinite integral}.
\\

{\em I would like to thank Tetsushi Ito for his encouragement and help throughout the preparation of this paper. }

\begin{notation}
\label{notation}
For a number field (resp.\, valuation field) $L$, we denote the ring of integers of $L$ (resp.\, valuation ring of $L$) by $\intring{L}$.  We denote an algebraic closure of $L$ by $\uesen{L}$.

Throughout the paper, we fix an odd prime $p$ and a totally real field $F$. We assume $p$ is {\em inert} in $F$. Let $\doi{p}$ be a unique prime ideal of $\intring{F}$ above $p$. Let $\adele_F$ be the ad\`ele ring of $F$ and let $\adele_{F,f}$ be the ring of finite ad\`eles.   Let $\cpxnum_p$ be the $p$-adic completion of $\uesen{\ratnum}_p$. We fix embeddings $F\hookrightarrow \uesen{\ratnum}$ and $\uesen{\ratnum}\hookrightarrow\cpxnum_p$. 
We denote by $\ord_p$, $\log_p$ the valuation map and the Iwasawa logarithm map respectively (they are normalized by $\ord_p(p)=1$ and $\log_p(p)=0$ respectively). We have a canonical decomposition \[\ratnum_p^\times\cong p^{\ratint}\times\bfw{F}_p^\times\times(1+p\ratint_p)\] and we denote by $\langle\cdot\rangle$ the projection $\ratnum_p^\times\rightarrow1+p\ratint_p$.  

\end{notation}

\section{Multiplicative integrals on $p$-adic upper half planes}
\label{secpadunif}
In the first half of this section, we summarize basics on $p$-adic measures on projective lines over non-archimedean local fields. We introduce the Bruhat-Tits tree for $\rmf{PGL}_2$ and multiplicative integrals following \cite{Das}.  They play a main role in this paper.  In the second half, following \cite{Das}, we review the $p$-adic uniformization theory for Mumford curves and their Jacobian varieties and define an important invariant, the $L$-invariants. 

In this section, let $K$ be a finite extention of $\ratnum_p$ , $\intring{K}$  the valuation ring with a uniformizer $\pi_K$, and $k_K$ the residue field. Recall that the $p$-adic completion of $\uesen{\ratnum}_p$ is denoted by $\C$.

\subsection{Basic notions of $p$-adic measures}
\label{psinmeasure}

\begin{defn}
The Bruhat-Tits tree $\ywrk{T}_K$ of ${\rm PGL}_2(K)$ is a graph whose vertex is the homothety class of $\intring{K}$-lattices in $K\oplus K$ and two vertices are connected by an edge if  there exist representatives of them such that one contains another and the quotient is isomorphic to $k_K$. We often identify the $\ywrk{T}_K$ with a geometrical realization as a simplicial complex.
\end{defn}

We denote by $\ywrk{V}(\ywrk{T}_K)$ (resp.\,$\ywrk{E}(\ywrk{T}_K)$) the set of the vertices  (resp.\,the oriented edges) of $\ywrk{T}_K$.  Let $v^*, w^* \in \ywrk{V}(\ywrk{T}_K)$ be the vertices which are the homothety class of $\intring{K}\oplus \intring{K}, \intring{K}\oplus \pi_K \intring{K}$ respectively, and let $e^*\in \ywrk{E}(\ywrk{T}_K)$ be the oriented edge from $w^*$ to $v^*$, denoted by $e^*=(w^*,v^*)$. For any oriented edge $e\in\ywrk{E}(\ywrk{T}_K)$, we denote the vertex of source (resp.\,target) by $s_e$ (resp.\, $t_e$), and we also denote by $\uesen{e}\in\ywrk{E}(\ywrk{T}_K)$ the oppositely oriented edge of $e$.   

For an oriented edge $e \in\ywrk{E}(\ywrk{T}_K)$, there exists  $\gamma \in {\rm PGL}_2(K)$ such that $e=\gamma e^*$. Then we assign to $e$ an open compact subset $\gamma \intring{K} \subset {\mathbb P}^1(K)$ (the action of ${\rm PGL_2(K)}$ on ${\mathbb P}^1(K)$ is given by the ${\rm M\ddot{o}bius}$ transformation) and denote it by $U_e$.  Note that $U_e$ is well-defined and independent of the choice of $\gamma$. The set $\{U_e\}_{e\in\ywrk{E}(\ywrk{T})}$ is an open basis of $\bfw{P}^1(K)$.

\begin{defn}
\label{end}
An {\em end} of $\ywrk{T}_K$ is an equivalence class of sequences $\{ v_n \}_{n\ge 0}$ of distinct vertices such that $(v_n, v_{n+1})$ is an oriented edge for all $n$. Here, two sequences $\{ v_n \}_{n\ge 0}$ and $\{ w_n \}_{n\ge 0}$ equivalent if there exists $n_0, k\in\ratint$ such that $v_i = w_{i+k}$ for all $i\ge n_0$.
\end{defn}
\begin{rmk}
\label{endbij}
There is a bijection between the set of ends of $\ywrk{T}_K$ and $\mathbb{P}^1(K)$ by the following correspondence:
\[
\{v_n\} \mapsto \bigcap_n U_{(v_n, v_{n+1})}
\]
\end{rmk}

\begin{defn}\label{H-valued measure}
Let $H$ be an abelian group and let $S$ be a non-empty subset of $\prjl{K}$.  An {\em $H$-valued measure} $\mu$ on $S$  is a map which assigns an element of $H$ to each open compact subset of $S$ with following two conditions:
\begin{enumerate}
\item $\mu(U \cup V)=\mu(U)+\mu(V)$ for disjoint open compact subsets $U,V\subset S$ 
\item $\mu (S)=0.$
\end{enumerate}
We denote by $\meas{S}{H}$  the space of H-valued measures on $S$.
\end{defn}

We put
\[E_{\ywrk{T}_K} \defeq \ds\bigoplus_e\ratint e \big{/}\bigoplus_e\ratint(e+\uesen{e})\]
 be a quotient of a free abelian group generated by oriented edges in $\ywrk{T}_K$ and let 
\[V_{\ywrk{T}_K} \defeq \ds\bigoplus_{v\in\ywrk{V}(\ywrk{T}_K)}\ratint v\]
be a free abelian group generated by vertices in $\ywrk{T}_K$.  We define a homomorphism $\rmf{Tr}$ by
\[\mapdiag{\rmf{Tr}:}{V_{\ywrk{T}_K}}{E_{\ywrk{T}_K}}{v}{\ds\sum_{s_e=v}e.}\]
Then we have
\begin{eqnarray}\meas{\prjl{K}}{H}=\rmf{Ker}(\rmf{Tr}^*),\label{Tr}\end{eqnarray}
where $\rmf{Tr}^*:\rmf{Hom}_\ratint\lt(E_{\ywrk{T}_K}, H\rt)\rightarrow\rmf{Hom}_\ratint\lt(V_{\ywrk{T}_K},H\rt)$ denotes a  homomorphism induced by $\rmf{Tr}$.

\begin{defn}
\label{univbruhattit}
We define a metric space $\ywrk{T}$ by
\[\ywrk{T} \defeq \dirlim{L/K:\mbox{\scriptsize fin ext}}\ywrk{T}_L\] 
where, for any finite extensions $L'/L/K$, $\ywrk{T}_L\rightarrow\ywrk{T}_{L'}$ is induced by $\otimes_{\intring{L}}\intring{L'}$.  We define sets $\ywrk{V}(\ywrk{T})$ and $\ywrk{E}(\ywrk{T})$ as follows:\begin{eqnarray}\ywrk{V}(\ywrk{T})&\defeq&\dirlim{L/K:\mbox{\scriptsize fin ext}}\ywrk{V}(\ywrk{T}_L)\subset\ywrk{T} \\ \ywrk{E}(\ywrk{T}) &\defeq&\bigcup_{L/K:\mbox{\scriptsize fin ext}}\ywrk{E}(\ywrk{T}_L) \end{eqnarray}
\label{metric}
For any two points $x,y \in \ywrk{T}$, we define a metric $d_\ywrk{T}$ as follows:
\[d_\ywrk{T}(x,y)\defeq \lim_{e_{L/K}\rightarrow\infty}\frac{1}{e_{L/K}}\#\lt\{v \in \ywrk{V}(\ywrk{T}_L)\relmiddle|v \mbox{ lies in the path from }x\mbox{ to }y \rt\},\]
where $e_{L/K}$ is the ramification index of $L/K$.
\end{defn}
\begin{rmk}
If both $x,y\in\ywrk{V}(\ywrk{T}_L)$ for some finite extension $L/K$, the distance $d_\ywrk{T}(x,y)$ is a rational number. 
\end{rmk}
\begin{rmk}
As in Definition \ref{end}, we can similarly define ends in $\ywrk{T}$, which is a infinitely long  path in the sense of the metric $d_\ywrk{T}$ . There is a similar  bijection as in Remark \ref{endbij} between $\prjl{\C}$ and the set of ends of $\ywrk{T}$.
\end{rmk}
From Definition \ref{univbruhattit}, we define two abelian groups $E_\ywrk{T}$ and $V_\ywrk{T}$:
\begin{eqnarray*}
E_\ywrk{T} \defeq \dirlim{L/K:\mbox{\scriptsize fin ext}}E_{\ywrk{T}_L} 
\\ V_\ywrk{T} \defeq \dirlim{L/K:\mbox{\scriptsize fin ext}}V_{\ywrk{T}_L}.
\end{eqnarray*}

Let $\ywrk{H}$ be the $p$-adic upper half plane, which is a rigid analytic space over $F$.  For any complete extension $F$ over $K$ inside $\C$, the set of $F$-valued points $\hfpl{F}$ coincides with  $\mathbb{P}^1(F) \setminus \prjl{K}$.  In particular, we have\[\hfpl{\C}=\prjl{\C}\setminus\prjl{K}.\]
\begin{defn}
We define the continuous map $\rmf{red}_K$ by
\[\mapdiag{\rmf{red}_K :}{\hfpl{\C}}{\ywrk{T}_K\subset\ywrk{T}}{[\tau]}{x_\tau,}\] 
where, $x_\tau$ is unique element in $\ywrk{V}(\ywrk{T})\cap\ywrk{T}_K$ such that $\{v_n\}_{n\ge 0}$ is a representative of the end corresponding to $\tau$ satisfying $v_0=x_\tau$ and $v_n \notin \ywrk{T}_K$ for any $n> 0$. We call $\red_K$ the reduction map to $K$.
\end{defn}

There exists a boundary map $\partial$:
\[\mapdiag{\partial:}{E_\ywrk{T}}{V_\ywrk{T}}{e=(s_e,t_e)}{t_e-s_e}.\]
and we have the following commutative diagram induced by $\red_K$:
\[
\xymatrix{
0 \ar[r] & \rmf{Div}_0\lt(\hfpl{\C}\rt) \ar[r] \ar[d]^{\rmf{red}_K}& \rmf{Div}\lt(\hfpl{\C}\rt) \ar[r] \ar[d]^{\rmf{red}_K}  & \ratint \ar@{=}[d]\ar[r] & 0 \\
0 \ar[r] & E_\ywrk{T} \ar[r]^\partial & V_\ywrk{T} \ar[r] & \ratint \ar[r] & 0.
}
\]

Let $H$ be a finitely generated free abelian group. Then we define an embedding
\[\rmf{Hom}_\ratint\lt(V_{\ywrk{T}_K},H\rt)\hookrightarrow\rmf{Hom}_\ratint\lt(\rmf{Div}\lt(\hfpl{\C}\rt),\ratnum\otimes_\ratint H\rt)\]
as follows:  Let $\phi \in \rmf{Hom}_\ratint\lt(V_{\ywrk{T}_K},H\rt)$, $\tau \in \hfpl{\C}$ and $e$ be an edge in $\ywrk{T}_K$ containing $\red_K(\tau)$. Then we define
\[\phi([\tau])\defeq d_\ywrk{T}(\red(\tau), s_e)\phi(t_e)+d_\ywrk{T}(\red(\tau), t_e)\phi(s_e).\]
Similarly we define an embedding
 \[\rmf{Hom}_\ratint\lt(E_{\ywrk{T}_K},H\rt)\hookrightarrow\rmf{Hom}_\ratint\big(\rmf{Div}_0\lt(\hfpl{\C}\rt), \ratnum\otimes_\ratint H\big)\] as follows:  Let $\psi \in \rmf{Hom}_\ratint\lt(E_{\ywrk{T}_K},H\rt)$. Then for any $\nami{e} \in \ywrk{E}(\ywrk{T})$ , if $\nami{e}\cap\ywrk{T}_K=\emptyset$ as sets, we define $\psi(\nami{e})=0$.  If $\nami{e}\cap\ywrk{T}_K\neq\emptyset$ as sets, there exists an oriented edge $e\in\ywrk{E}(\ywrk{T}_K)$ which contain $\nami{e}$ as sets and same direction as $e$. Then we define
\[
\phi(\nami{e})\defeq d_\ywrk{T}(s_{\nami{e}},t_{\nami{e}})\phi(e).
\]
Via the reduction map, we regard $\psi$ as an element in  $\rmf{Hom}_\ratint\big(\rmf{Div}_0\lt(\hfpl{\C}\rt), \ratnum\otimes_\ratint H\big)$.

\subsection{$p$-adic multiplicative integrals}
\label{batusekibun}
In this section, we define multiplicative integrals for $H$-valued measures on $\prjl{K}$ following \cite{Das}. Proposition \ref{ordtobatuint} in this section implies that the multiplicative integrals have more information than usual integrals.

\begin{defn}
Let $F/K$ be a complete extension contained in $\C$. Let H be a finitely generated free abelian group,  $S$ be a subset in $\prjl{K}$, $d\in {\rm Div_0}\left(\prjl{F}\setminus S\right) $ a divisor of degree $0$, 
and let $\mu\in\meas{S}{H}$ be a $H$-valued additive measure on $S$. Then we define:

\begin{eqnarray*}
\batuint{d}{\omega_\mu} & \defeq &\batuint{d}{f_d (t)} d\mu(t)   \\*[3mm]
& := & \ds\lim_{||\ywrk{U}||\rightarrow 0 }\prod_{U\in \ywrk{U}} f_d(t_U)\otimes_\ratint\mu(U)  \in F^{\times} \otimes_\ratint H  
\end{eqnarray*}
where $\ywrk{U}$ is an open compact disjoint covering of $S$,  $||\ywrk{U}||$ is the  supremum of the diameter of $U$ for $U\in \ywrk{U}$,  $f_d$ is a rational function on $\mathbb{P}^1(F)$ whose divisor is $d$ , and  $t_U$ is an element in $U$.
\end{defn}
\begin{rmk}
Because of the properties of $\mu$ in Definition \ref{H-valued measure}, the middle term of
the above formula is independent of the choice of $f_d$. So the definition is well-defined.
\end{rmk}
If $S$ is invariant under the action of a subgroup $\Gamma\subset\rmf{PGL}_2(K)$, we define the action of $\gamma\in\Gamma$ on $\mu\in\meas{S}{H}$ by \[\gamma\cdot\mu(U)=\mu(\gamma^{-1} U)\] for any open compact subset $U\subset S$. 
\begin{prop}
\label{ordtobatuint}
Let H be a finitely generated free abelian group, and let $\mu$ be an H-valued measure on $\prjl{K}$.
Let ${\rm ord}: \C \rightarrow \ratnum$ be the valuation such that $\rmf{ord}(\pi_K)=1$. Then for any $\tau_1, \tau_2 \in \hfpl{\C} $
\[ 
({\rm ord}\otimes_\ratint {\rm id}_H)\left( \batuint{[\tau_1]-[\tau_2]}{\omega_\mu} \right)= \mu\big([\tau_1]-[\tau_2]\big),
\]
where $\mu$ is regarded as an element in $\rmf{Hom}_\ratint\big(\rmf{Div}_0\lt(\hfpl{\C}\big),H\rt)$ as in Section \ref{psinmeasure} (see the equality (\ref{Tr}) and the end of Section \ref{psinmeasure}).
\end{prop}
\begin{proof}
We may assume $\tau_1, \tau_2 \in \hfpl{L}$ for some finite extension $L/K$. We denote $e_{L/K}$ the ramified index of $L/K$.  By partition the section from $\tau_2$ to $\tau_1$ and by replacing $\mu$ with $\gamma\mu$ for suitable $\gamma \in \rmf{PGL}_2(K)$, we may also assume both $\red_K(\tau_1)$ and $\red_K(\tau_2)$ lie in the edge $e^*$.

The elements $\tau_i$ have the following expansion:
\[\tau_i=\pi_L^{-a_i}u_i+x_i\]
where $u\in \intring{L}^\times$ such that if $a_i=0$, the image in $k_L$ is not contained in $k_K$, $x_i \in \intring{L}$ and $0 \le a_i < e_{L/K}$.

Then we have
\[
\ord(x-\tau_i)=
\begin{cases}
\ord(x) & (x\in \prjl{K}\setminus\intring{K}
) \\
\ds-\frac{a_i}{e_{L/K}} & (x\in \intring{K})   
\end{cases}
\]     
Thus
\[
\ord\lt(\frac{x-\tau_1}{x-\tau_2}\rt)=
\begin{cases}
0 & (x\in \prjl{K}\setminus\intring{K}
) \\
\ds\frac{a_2}{e_{L/K}}-\frac{a_1}{e_{L/K}} & (x\in \intring{K})   
\end{cases}
\]     
Therefore
\[({\rm ord}\otimes_\ratint {\rm id}_H)\left( \batuint{[\tau_1]-[\tau_2]}{\omega_\mu} \right)=\lt(\frac{a_2}{e_{L/K}}-\frac{a_1}{e_{L/K}}\rt)\mu(\intring{K})\]
but the right hand side is $\mu([\tau_1]-[\tau_2])$

\qed\end{proof}

\subsection{The Manin-Drinfeld theorem}
\label{The Manin-drinfeld theorem}
Let $\Gamma \subset {\rm PGL}_2(K)$ be a discrete finitely generated subgroup without torsion, and let $\ywrk{L}_\Gamma$ be the set of the $Q \in \prjl{K}$ such that there exists $P \in \mathbb{ P}^1(K)$ and a sequence $\{\gamma_n\}_{n\in \bfw{N}}\subset \Gamma$ consisting of distinct elements such that $\gamma_n P \rightarrow Q\ (n \rightarrow \infty)$.
\begin{defn}
The Bruhat-Tits tree of $\Gamma$ is the subtree $\ywrk{T}_{\Gamma} \subset \ywrk{T}_K$ defined by
\[ 
\ywrk{T}_{\Gamma} \defeq \bigcup_{\substack{\ywrk{L}_\Gamma\cap U_e \neq\emptyset \\ \ywrk{L}_\Gamma\setminus U_{e}\neq\emptyset}}e 
\]
\end{defn}

Denote by $\meas{\ywrk{L}_\Gamma}{H}^\Gamma$ the $\Gamma$-invariant part of $\meas{\ywrk{L}_\Gamma}{H}$. Then we have:
\begin{eqnarray*}
\meas{\ywrk{L}_\Gamma}{H}^\Gamma & = &{\rm Ker}\left({\rm Hom}_\ratint\left(E_{\ywrk{T}_\Gamma}, H\right)\overset{\rm Tr^*}{\rightarrow}{\rm Hom}_\ratint\left(V_{\ywrk{T}_\Gamma}, H\right)\right)^\Gamma \nonumber \\
& = & {\rm Ker}\left({\rm Hom}_\ratint\left(\left(E_{\ywrk{T}_\Gamma}\right)_\Gamma, H\right)\overset{\rm Tr^*}{\rightarrow}{\rm Hom}_\ratint\lt(\left(V_{\ywrk{T}_\Gamma}\right)_\Gamma, H\right)\rt) \\ \label{measurenosiki}
&=& \rmf{Hom}_\ratint\big( \rmf{Coker}\left(\rmf{Tr} \right)_\Gamma, H \big), \label{caniso}
\end{eqnarray*}
where $(V_{\ywrk{T}_\Gamma})_\Gamma$, $(E_{\ywrk{T}_\Gamma})_\Gamma$ are the maximal $\Gamma$-invariant quotient of $V_{\ywrk{T}_\Gamma}$, $E_{\ywrk{T}_\Gamma}$ and they are equal to the corresponding abelian groups of the graph $\Gamma\backslash\ywrk{T}_\Gamma$.

\begin{rmk}
\label{batuintnohuhensei}
If $\mu$ is $\Gamma$-invariant, $\batuint{\gamma d}{\omega_\mu}=\batuint{d}{\omega_\mu} \mbox{\rm for any }\gamma \in \Gamma$.
\end{rmk}

It is known that the the quotient $\Gamma\backslash\ywrk{T}_\Gamma$ is a finite graph, and $\rmf{Coker}(\rmf{Tr})_\Gamma$ is isomorphic to $H^1(\Gamma, \ratint)$ (see \cite{Das}, section 2.3).

We have $\mu\in\meas{\ywrk{L}_\Gamma}{H^1\left(\Gamma, \ratint\right)}$ corresponding to \[{\rm id}_{H^1\left(\Gamma, \ratint\right)} \in {\rm Hom}\left(H^1\left(\Gamma, \ratint\right), H^1\left(\Gamma, \ratint\right)\right).\] This measure $\mu$ can be described explicitly:  We fix a vertex $v \in \ywrk{T}_\Gamma$.  Let $e_1, e_2, \dots e_n$ be a set of edges of $\Gamma\bbs\ywrk{T}_\Gamma$ (we fix orients of them). For any $\gamma \in \Gamma$, let $e_\gamma$ be an element in $E_{\ywrk{T}_\Gamma}$ such that $\partial(e_\gamma)=\gamma v-v$.   Write $e_\gamma=m_1 e_1+\dots +m_n e_n$ and $e_\gamma^*:= m_1 e_1^* + \dots +m_n e_n^*$ where $e_j^*$ is the dual of $e_j$.  Then $\mu$ is as follows:
\[
\mu(U_e)(\gamma)=e_\gamma^*(e)
\]

Roughly speaking, the value of $\mu(U_e)(\gamma)$ is the number including orient of $e$ lying in the path from $v$ to $\gamma v$ modulo $\Gamma$.

\begin{thm}[Mumford]
\label{Mumford's theorem}
Let $X$ be a Mumford curve over $K$ ({\em i.e.} stable reduction of $X$ contains only rational curves that intersect at normal crossing over $k_K$). Then there exists a subgroup $\Gamma\subset{\rm PGL}_2(K)$ and an $\rmf{Aut}(\C/K)$-equivariant rigid analytic isomorphism:
\[
X(\C) \cong \Gamma\backslash\padhfpl{\C}.
\]
Moreover $\Gamma$ is discrete, free of rank $g$ ($g$ is the genus of $X$) and unique up to conjugation in ${\rm PGL}_2(K)$. 
\end{thm}

\begin{rmk}
For the Mumford curves appearing in the $p$-adic uniformization of Shimura curves, we have $\ywrk{L}_\Gamma=\mathbb{P}^1(K)$.
\end{rmk}

Let $\Gamma$ be as in Theorem \ref{Mumford's theorem}. Then $H\defeq\rmf{Coker}(\rmf{Tr})_\Gamma$ is finitely generated free abelian group.  Let $\mu$ be  the universal $\Gamma$-invariant measure in $\meas{\ywrk{L}_\Gamma}{H}^\Gamma$,  corresponding to $\rmf{id} \in \rmf{End}_\ratint(H)$.

At first, we construct a $\ratint$-lattice $\Lambda$ in $\C^\times \otimes_\ratint H$.  By the following exact sequence:

\[
\xymatrix{
0 \ar[r] & {\rm Div}_0\left(\padhfpl{\C}\right) \ar[r] & {\rm Div}\left(\padhfpl{\C}\right)\ar[r] & \ratint \ar[r] & 0,\\
}
\]
we have a homomorphism \[\delta:H_1(\Gamma, \ratint) \rightarrow {\rm Div}_0\lt(\padhfpl{\C}\right)_\Gamma\] by the long exact sequence of the group homology (we denote $H_0(\Gamma,\ \cdot\ )$ by $(\ \cdot\ )_\Gamma$).  Since there is no stabilizers for an element of $\hfpl{\C}$ (see \cite{GvdP}, p.7, Proposition (1.6.4)), we see that $H_1\left(\Gamma, {\rm Div}\left(\padhfpl{\C}\right)\right)=0$. Thus $\delta$ is injective.

On the other hand, the map
\[{\rm Div}_0\left(\padhfpl{\C}\right)_\Gamma\ni d \mapsto \batuint{d}{\omega_\mu} \in \C^\times \otimes H\]
is a well-defined homomorphism (see Remark \ref{batuintnohuhensei}).  Therefore we define the lattice $\Lambda$ by the image of $\batuint{\ }{\omega_\mu}\circ \delta$.  Now we state the following theorem of Manin-Drinfeld describing the $p$-adic uniformization of the Jacobian variety of a Mumford curve.

\begin{thm}[Manin-Drinfeld]
\label{manindrinfeld}
The morphism
\[
\begin{array}{ccc}
J(X)(\C) & \rightarrow & (\C^\times \otimes H) / \Lambda \\
\rotatebox{90}{$\in$}  &  & \rotatebox{90}{$\in$} \\
\tilde{x}-\tilde{y} & \mapsto & \batuint{\tilde{x}-\tilde{y}}{\omega_\mu} \\
\end{array}
\]
is an ${\rm Aut}(\C/K)$-equivariant rigid analytic isomorphism.
\end{thm}
\begin{proof}
See \cite{Das}, Theorem 2.5. \qed
\end{proof}

\subsection{$L$-invariants}
\label{L-invariants}
In this section, we construct $L$-invariants associated to a homomorphism $\phi:\C^\times \rightarrow \C$ following \cite{Das}. Then for each $\phi$, we define a map from  $J(X)(\C)$ to $\homm\left(\meas{\ywrk{L}_\Gamma}{\C}^\Gamma, \C\right)$ and determines the value in $\C$ for a $\Gamma$-invariant $\C$-valued measure.

We rewrite the statement of Theorem \ref{manindrinfeld} as the following exact sequence:

\[
\xymatrix{
0 \ar[r] &H_1(\Gamma, \ratint) \ar@{}[d]|{\rotatebox{90}{$\cong$}} \ar[r]^(0.4)j&
\C^\times \otimes \rmf{Coker}(\rmf{Tr})_\Gamma \ar[r] \ar@{}[d]|{\rotatebox{90}{$\cong$}} & J(X)(\C) \ar[r] & 0 \\
& \Gamma^{ab} & \homm_\ratint\left(\Gamma^{ab}, \C^\times\right)& \\
& \gamma \ar@{}[u]|{\rotatebox{90}{$\in$}} \ar@{|->}[r] & \left[ \gamma' \mapsto \batuint{\gamma v-v}{\omega_\mu}(\gamma')\right] \ar@{}[u]|{\rotatebox{90}{$\in$}} & \\
}
\]

Fix a homomorphism $\phi:\C^\times \rightarrow \C$.  Let ${\rm ord}:\C^\times \rightarrow \ratnum\subset\C$ be the valuation map with ${\rm ord}(\pi_K)=1$.  From the explicit description of $\mu$ (see the end of Section \ref{psinmeasure}) and Proposition \ref{ordtobatuint}, the pairing:
\[
\xymatrix{
\Gamma^{ab} \times \Gamma^{ab} \ar[r] & \C \\
(\gamma, \gamma') \ar@{}[u]|{\mws{$\in$}} \ar@{|->}[r] & {\rm ord}\left( \batuint{\gamma v-v}{\omega_\mu}(\gamma')\right) \ar@{}[u]|{\mws{$\in$}}\\
} 
\]
is symmetric and non-degenerate.  Therefore the composition of $j$ in the above diagram and  $\ord\otimes_\ratint \rmf{id}$ induces an isomorphism\[H_1(\Gamma, \C)\overset{\cong}{\longrightarrow}\C\otimes_\ratnum\rmf{Coker}(\rmf{Tr})_\Gamma.\]  The map $(\phi\otimes_\ratint \rmf{id}) \circ j$ induces a similar homomorphism.  
\begin{defn}
For any homomorphism $\phi:\C^\times \rightarrow \C$, the {\em $L$-invariant} is a unique endomorphism:
\[\mathscr{L}_\phi\in {\rm End}\left(\C\otimes_\ratint\rmf{Coker}(\rmf{Tr})_\Gamma\right)\] such that:
\[
\lt(\phi\otimes_\ratint \rmf{id} - \mathscr{L}_\phi\circ (\ord\otimes_\ratint \rmf{id}) \rt)\circ j =0.
\]
\end{defn}
From the above exact sequence, we define a well-defined homomorphism $\phi^X$:
\[\mapdiag{\phi^X:}{J(X)(\C)}{\C\otimes_\ratint \rmf{Coker}(\rmf{Tr})_\Gamma\cong\rmf{Hom}_\ratint\left(\rmf{Coker}\left(\rmf{Tr}\right), \C\right)^*}{P}{\phi\otimes_\ratint \rmf{id}(\nami{P}) - \mathscr{L}_\phi\circ (\ord\otimes_\ratint \rmf{id})(\nami{P}),}\] where $\nami{P}$ is a lift in $\C^\times\otimes_\ratint\rmf{Coker}(\rmf{Tr})_\Gamma$ and * means the $\C$-linear dual. 

The abelian group $\rmf{Hom}_\ratint\left(\rmf{Coker}_\Gamma\left(\rmf{Tr}\right), \C\right)$ is isomorphic to $\meas{\ywrk{L}_\Gamma}{\C}^\Gamma$ (see the equality (\ref{caniso}) in Section \ref{The Manin-drinfeld theorem}).  Thus for any $P\in J(X)$, $\phi^X(P)$ gives values in $\C$ for each $\Gamma$-invariant $\C$-valued measure.

\section{Automorphic forms on definite quaternion algebras and Hida families}
\label{Quaternionic automorphic forms}
\subsection{Basic definitions}
\label{Basic definitions}

For any algebraic group $G$ over $F$, we denote by $G(\adele_F)$ and $G(\adele_{F,f})$ the $\adele_F$ and $\adele_{F,f}$-valued points respectively. For any $x\in G(\adele_F)$, we denote by $x_v\in G(F_v)$ the $v$-component of $x$ for any place $v$ of $F$ and for any subgroup $H=\prod H_v \subset G(\adele_F)$ we denote $H^v$ the subset in $H$ which consists of $h\in H$ such that $h_v=1$.


Let $F$ be a totally real field and $\doi{p}$ the unique prime ideal above the odd prime number $p$ as in Notation \ref{notation}  Let $B$ be a definite quaternion algebra over $F$, which is ramified at all archimedean places. Let $\doi{n}^-$ be the product of the finite prime ideal where $B$ ramified at. We assume that $B$ is split at $\doi{p}$. We denote by $\yama{B}^\times$ the group $B^\times(\adele_{F,f})$, which is the group of finite ad\`elic points of $B^\times$. 

Let $\doi{a}$ be a non-zero ideal of $\intring{F}$ which is relatively prime to $\doi{n}^-$.  For any prime $\doi{l}$, let $R(\doi{a})_\doi{l}\subset B_\doi{l}:=B\otimes_F F_\doi{l}$ be as follows:
\[
R(\doi{a})_\doi{l}\defeq 
\begin{cases}
\mbox{the unique maximal order of }B_\doi{l} & \mbox{if } \doi{l}\mid\doi{n}^- \\
\mbox{an Eichler order of level } \doi{a}\intring{F_\doi{l}} & \mbox{if } \doi{l}\nmid\doi{n}^- . 
\end{cases}
\]

Let $\widehat{R}(\doi{a}) \defeq \ds\prod_\doi{l} R(\doi{a})_\doi{l}\subset B\otimes\adele_{F,f}$, and $R(\doi{a}) \defeq B \cap \widehat{R}(\doi{a})$, which is called an Eichler order of $B$ of level $\doi{a}$.

For each prime $\doi{l} \nmid \doi{n}^-$, we fix an isomorphism of $F_\doi{l}$-algebras
\[
\iota_\doi{l}:B_\doi{l}=B\otimes_F F_\doi{l} \overset{\cong}{\longrightarrow} M_2(F_\doi{l})
\]
(for any ring $A$, $\rmf{M}_2(A)$ is the ring of $2\times 2$ matrices with entries in $A$).  The map $\iota_{\doi{l}}$ induces an isomorphism between ${B_\doi{l}}^\times$ and ${\rm GL}_2(F_\doi{l})$.  By exchanging $\iota_\doi{l}$ for its conjugation,  we may assume that the image of $R(\doi{a})_\doi{l}$ is \[\left\{\gyoretu{a}{b}{c}{d} \in \rmf{M}_2\left(\intring{F_\doi{l}}\right) \relmiddle| c\equiv 0 \mod  \doi{a}\intring{F_\doi{l}}\right\}.\]

\begin{defn}
\label{leveldayo}
For any non-zero ideal $\doi{a}$ of $\intring{F}$ which is relatively prime to $\doi{n}^-$.  Then we put: 

\begin{eqnarray*}
\Sigma_0(\doi{a}, \doi{n}^-) &\defeq & \widehat{R}(\doi{a})^\times \\
\Sigma_1(\doi{a}, \doi{n}^-)& \defeq & \left\{ u \in \Sigma_0(\doi{a},\doi{n}^-) \relmiddle| \iota_\doi{l}(u) \equiv \gyoretu{1}{*}{0}{*} \mod  \doi{a}\rmf{M}_2\left(\intring{F_\doi{l}}\right)\mbox{\rm\  for }\doi{l}\nmid \doi{n}^-\right\} \\
\Delta_0\left(\doi{a},\doi{n}^-\right)&\defeq &\yama{R}(\doi{a}) \\
\Delta_1\left(\doi{a},\doi{n}^-\right) & \defeq & \left\{ x \in \Delta_0(\doi{a},\doi{n}^-) \relmiddle| \iota_\doi{l}(u) \equiv \gyoretu{1}{*}{0}{*} \mod  \doi{a}\rmf{M}_2\left(\intring{F_\doi{l}}\right)\mbox{\rm\  for }\doi{l}\nmid \doi{n}^-\right\}
\end{eqnarray*}
\end{defn}

\begin{defn}
\label{hokeikeisiki}
Let $\Sigma$ be an open compact subgroup of $\yama{B}^\times$ and let $M$ be a $\ratint_p$ -module with $\iota_\pee(\Sigma_\pee)$ action.  An {\em $M$-valued automorphic form on $\yama{B}^\times$ of level $\Sigma$} is a function
\[ \Phi:\yama{B}^\times \longrightarrow M \]
such that
\[ \Phi(\gamma b u)=\iota_\pee(u_\pee)^{-1}\cdot\Phi(b) \]
for all $\gamma \in B^\times, b\in\yama{B}^\times, u\in\Sigma$.  We denote by $S(\Sigma, M)$ the {\em space of $M$-valued automorphic forms on $\yama{B}^\times$ of level $\Sigma$}.
\end{defn}

\begin{rmk}
Since $\dblcoset{$B^\times$}{$\yama{B}^\times$}{$\Sigma$}$ is a finite set, $\Phi$ is determined by its values on a finite set of representatives of the double coset space.   
\end{rmk}

\begin{defn}
\label{saijitakousiki}
For each embedding $\sigma:F_p \rightarrow \uesen{\ratnum}_p$, and any $n\ge0$, let ${\rm Sym}^{n}$ be the $\cpxnum_p$-vector space of homogeneous polynomials of degree $n$ in the indeterminates $X_\sigma, Y_\sigma$ with coefficients in $\cpxnum_p$.  We put \[
\mathcal{B}_{n} \defeq \bigotimes_{\sigma:F_p \rightarrow \uesen{\ratnum}_p} {\rm Sym}^{n_\sigma}.\]
We define a right action of $\rmf{GL}_2(F_p)$ on $\ywrk{B}_{n}$: 
\[
\bigotimes_\sigma P^\sigma\lt(X_\sigma, Y_\sigma\rt)|\gamma \defeq \bigotimes_\sigma P^\sigma\lt(a^\sigma X_\sigma+b^\sigma Y_\sigma, c^\sigma X_\sigma+d^\sigma Y_\sigma\rt)
\]
for $\gamma=\gyoretu{a}{b}{c}{d} \in \rmf{GL}_2(F_\pee)$ and $P^\sigma\lt(X^\sigma, Y^\sigma\rt)\in \rmf{Sym}^{n}$.
Then we put
\[ V_{n} \defeq \rmf{Hom}_{\cpxnum_p}\lt(\mathcal{B}_{n}, \cpxnum_p\rt) \]
with the left action of $\rmf{GL}_2(F_\spp)$ induced by $\ywrk{B}_{n}$.  For any $k\ge 2$, we call $S(\Sigma, V_{k-2})$ {\em the space of classical automorphic forms on $\yama{B}^\times$ of weight $k$, and the level $\Sigma$}.
\end{defn}

We consider the following action of $\widehat{F}^\times$ on $S(\Sigma, V_{k})$ by
\[
z\cdot\Phi(b)\defeq \Phi(zb).
\]
This action factors through the infinite idele class group
\[
Z_F(\Sigma)\defeq \widehat{F}^\times / {F_+^\times(\widehat{\intring{F}}^\times \cap \Sigma)^\pee}.
\]
We have a natural surjection from $Z_F(\Sigma)$ to a finite idele class group ${\it Cl}(\Sigma)$
\[
Z_F(\Sigma)\rightarrow {\it Cl}(\Sigma) :=\widehat{F}^\times /F_+^\times(\widehat{\intring{F}}^\times \cap \Sigma)
\]
whose kernel is given by the image of $\intring{F_\pee}^\times \cap \Sigma_\spp$ in $Z_F(\Sigma)$.

Let $\chi_{F,\rmf{cycl}}$ be the restriction of the cyclotomic character to $\rmf{Gal}(\uesen{F}/F)$.
By Definition \ref{hokeikeisiki} and \ref{saijitakousiki}, the action of $\intring{F_\pee}^\times \cap \Sigma_\spp$ is given by multiplying $\chi_{F,{\rm cycl}}^{k-2}(z)$ for each $z \in \intring{F_\pee}^\times \cap \Sigma_\spp$.

\begin{defn}
For each character $\eta$ of ${\it Cl}(\Sigma)$, we define
\[
S(\Sigma, V_{k}, \eta)\defeq\left\{ \Phi \in S(\Sigma, V_{k}) \relmiddle|
\begin{aligned}
& \Phi(zb)=\chi_{F,{\rm cycl}}^{k-2}(z) \, \eta(z)^{-1}\Phi(b) \\
& \mbox{\rm for all}\ z \in \widehat{F}^\times,\,b\in \widehat{B}^\times
\end{aligned}
\right\}.
\]

\end{defn}
We have a decomposition:
\[
S(\Sigma, V_{k})=\bigoplus_\eta S(\Sigma, V_{k}, \eta),
\]
where $\eta:{\it Cl}(\Sigma)\rightarrow\cpxnum_p^\times$ runs over the characters of ${\it Cl}(\Sigma)$.  

We recall the definition of Hecke operators. 
\begin{defn}
\label{Hecke operators}
Let $\doi{n}\subset \intring{F}$ be a nonzero prime ideal which is relatively prime to $\doi{n}^-$. 
There exist two kinds of operators $T(\doi{a}), T(\doi{a},\doi{a})$ for certain non-zero ideals $\doi{a} \subset \intring{F}$ acting on the space of automorphic forms.  Let $(\Delta,\Sigma)=(\Delta_0,\Sigma_0)$ or $(\Delta_1,\Sigma_1)$.
\begin{itemize}
\item{\em (Definition of $T(\doi{a}))$:}  For any $(\doi{a},\doi{n})=1$,  given the right coset decomposition
\begin{eqnarray}
\label{migikoset}
\left\{ x \in \Delta(\doi{n}, \doi{n}^-) \relmiddle| \rmf{Nrd}_{B/F}(x)\intring{F}=\doi{a} \right\}=\bigsqcup_i \sigma_i \Sigma(\doi{n}, \doi{n}^-).
\end{eqnarray}
Let $M$ be a $\ratint_p$-module as in Definition \ref{hokeikeisiki} with the action of $\Delta(\doi{n}, \doi{n}^-)$ which is compatible with that of $\iota_\doi{l}(\Sigma(\doi{n},\doi{n}^-))$. For any $\Phi \in S(\Sigma(\doi{n},\doi{n}^-), M)$, we define
\[
\left(T(\doi{a})\Phi\right)(b) \defeq \displaystyle \sum_i \sigma_i\Phi(b\sigma_i).
\]
\item {\em (Definition of $T({\doi{a},\doi{a}})$):} For any $(\doi{a},\doi{n}\doi{n}^-)=1$, let $a\in \adele_{F, f}$ be an element such that $a\intring{F}=\doi{a}$. Let $M$ be a $\ratint_p$-module as above. For any $\Phi \in S(\Sigma(\doi{n},\doi{n}^-), M)$, we define
\[
\left(T(\doi{a},\doi{a})\Phi\right)(b) \defeq \Phi(ba).
\]


\end{itemize}
\end{defn}

\begin{rmk}
When $\doi{a}=\spp$ and $\doi{n}$ is prime to $\spp$, th Hecke operator $T(\spp)$  is described explicitly. The right coset decomposition of (\ref{migikoset}) is given as follows (see \cite{Shi1}, Proposition 3.36):
\begin{eqnarray}
&&\mbox{if $\Sigma=\Sigma_0$, }\ds\bigsqcup_{b\in \intring{F}/\doi{p}} \gyoretu{\pi_\doi{p}}{b}{0}{1}\Sigma_0(\doi{n}, \doi{n}^-)\sqcup\gyoretu{1}{0}{0}{\pi_\spp}\Sigma_0(\doi{n}, \doi{n}^-)\label{migikoset1}, \\
&&\mbox{if $\Sigma=\Sigma_1$, }\ds\bigsqcup_{c\in\intring{F}/\doi{p}} \gyoretu{1}{0}{\pi_\doi{p}c}{\pi_\doi{p}}\Sigma_1(\spp\doi{n}, \doi{n}^-) \label{migikoset2},
\end{eqnarray}
where $\pi_\doi{p}\in\intring{F_\doi{p}}$ is a uniformizer. 
\end{rmk}

\subsection{Quaternionic automorphic forms and the Bruhat-Tits tree}
\label{another description} 
Let $\Sigma:=\Sigma_0(\doi{a}, \doi{n}^-)$ for a non-zero ideal $\doi{a}$ relatively prime to $\doi{n}^-$.  By the strong approximation theorem, the reduced norm map $\rmf{Nrd}_{B/F}$ gives a bijection:
\[
\xymatrix{
 \rmf{Nrd}_{B/F} :B^\times\bbs\yama{B}^\times\big/B_\spp^\times\Sigma \ar[r]& F_+^\times\bbs\adele_{F, f}^\times\big/\yama{\intring{F}}^\times =:\nnm{Cl}_F^+ 
},
\]
where $\nnm{Cl}_F^+$ is the narrow ideal class group.  Note that since we have assumed $p$ is inert in $F$, the image of $F_\doi{p}^\times$ in $\nnm{Cl}_F^+$ is trivial.  We have a decomposition
\begin{eqnarray} 
\label{Bdblcoset}
\yama{B}^\times = \bigsqcup_{i=1}^h B^\times x_i B_\doi{p}^\times\Sigma,
\end{eqnarray}
where $h=\#\nnm{Cl}_F^+$ is the narrow class number of $F$ and the elements $x_i \in \yama{B}^\times$ satisfies $(x_i)_\doi{p}=1$ and the images of $x_1,\dots,x_h$ by $\rmf{Nrd}_{B/F}$ give a set of complete representatives of the finite group $\nnm{Cl}_F^+$.  

For $i=1, \dots ,h$, we define
\begin{eqnarray}
\nami{\Gamma}_i=\nami{\Gamma}_i(\doi{a}, \doi{n}^-) & \defeq & B^\times \cap x_i \yama{B}_p^\times \Sigma x_i^{-1}, \\ \label{namigammai}
\Gamma_i= \Gamma_i(\doi{a}, \doi{n}^-) &\defeq & \left\{\gamma \in \nami{\Gamma}_i \relmiddle| \rmf{Nrd}_{B/F}(\gamma) \in U_{F, +} \right\}, \label{gammai}
\end{eqnarray} 
where $U_{F,+}$ is the set of totally positive units.  Using (\ref{Bdblcoset}), we have a bijection
\begin{eqnarray}
\label{bij1}
\xymatrix{
\displaystyle\bigsqcup_{i=1}^h\dblcoset{$\nami{\Gamma}_i$}{$\yama{B}_p^\times$}{$\Sigma_p$} \ar[r]^(0.5)\cong & \dblcoset{$B^\times$}{$\yama{B}^\times$}{$B_\doi{p
}^\times\Sigma$},
}
\end{eqnarray}
which sends $g \in \dblcoset{$\nami{\Gamma}_j$}{$\yama{B}_p$}{$\Sigma_p$}$ to $x_j g$. 
      
By (\ref{bij1}), an $M$-valued automorphic form $\Phi \in S(\Sigma, M)$ (where $\Sigma\subset\yama{B}^\times $is an open compact subgroup and $M$ is a $\ratint_p$-module as in Definition \ref{hokeikeisiki}) can be defined as an $h$-tuple of functions $\phi^1, \dots ,\phi^h$ on $\rmf{GL}_2(F_\spp)$ by the rule $\phi^i(g) = \Phi(x_i g)$ for $i=1,\dots , h$. These functions $\phi^i$ satisfy
\begin{eqnarray}
\label{hokeiseidayo}
\phi^i(\gamma gu)=u^{-1}\phi^i(g)
\end{eqnarray}
for $\gamma \in \nami{\Gamma}_i, g\in\rmf{GL}_2(F_p), u\in\Sigma_p.$

Now we give another description of quaternionic automorphic forms on $\yama{B}^\times$ in terms of latices and the Bruhat-Tits tree.  Let $\Phi \in S(\Sigma, M)$ be an $M$-valued automorphic form, and let $(\phi^1, \dots ,\phi^h)$ be an $h$-tuple attached to $\Phi$ as above.  
There exists a bijection
\[ \xymatrix{{\xi_\doi{a}:}\hspace{-12mm}&\ \yama{B}_\spp^\times\big/\Sigma \ar[r]^(0.176){1:1} & \ywrk{P}(\doi{a}) \defeq\left\{(L_1,L_2)\relmiddle|{\begin{aligned}
&L_1, L_2  \subset F_\spp^2 \\
&\mbox{are lattices s.t.}\ L_1/L_2 \cong \intring{F}/\doi{p}^{\ord_\doi{p}{\doi{a}}}
\end{aligned}}
\right\} \\ & g \ar@{}[u]|{\mws{$\in$}} \ar@{|->}[r] & \left(g(\intring{F_p}\times\intring{F_p}),\ g(\intring{F_p}\times \doi{a}\intring{F_p})\right) \ar@{}[u]|{\mws{$\in$}}.}\]
\begin{defn}
\label{latticenofunction}
For $i=1, \dots,h$ and for $(L_1,L_2)\in\ywrk{P}(\doi{a})$  we define
\[ c_{\phi^i}(L_1,L_2) \defeq g\phi^i(g), \]
where $g(\intring{F_p}\times\intring{F_p})=L_1$ and $g(\intring{F_p}\times \doi{a}\intring{F_p})=L_2$.
\end{defn}
\begin{rmk}
By the formula (\ref{hokeiseidayo}), the function $c_{\phi^i}$ has the following property:
\[c_{\phi^i}(\gamma L_1, \gamma L_2) =\gamma c_{\phi^i}(L_1, L_2) \]
for all $\gamma\in\nami{\Gamma}_i(\doi{a}, \doi{n}^-)$.
\end{rmk}



\subsection{Measure valued automorphic forms}

Fix a valuation ring $\fxdring\subset\cpxnum_p$ finite flat over $\ratint_p$ containing all conjugates of $\intring{F}$. Let $\doi{n}^+$ be an ideal relatively prime to $\spp\doi{n}^-$ and let $\Sigma:=\Sigma(\doi{n}^+, \doi{n}^-)$ be an open compact subgroup of $\yama{B}^\times$.  We denote by $Z_{F,0}$ the kernel of the homomorphism from $Z_F(\Sigma)$ to $\nnm{Cl}_F(\Sigma).$
Explicitly, $Z_{F, 0} = \intring{F_\spp}^\times \big{\slash} \posuni$, where $\posuni$ is the closure of the set of totally positive units of $\intring{F}$ in $\intring{F_\spp}^\times$.  

We define several rings as follows:

\begin{eqnarray*}
\widetilde{\Lambda}_F &\defeq& \fxdring{}[[Z_{F,0}]], \\
\widetilde{\Lambda}_\ratnum &\defeq& \fxdring[[\ratint_p^\times]], \\
\Lambda  &\defeq& \fxdring[[1+\ratint_p]], \\
\Lambda^\dagger &\defeq& \cpxnum_p\left\dlangle T-2 \right\drangle. \\
\end{eqnarray*}
Here, $\widetilde{\Lambda}_F$, $\nami{\Lambda}_\ratnum$, $\Lambda$ are the completed group algebras and $\Lambda^\dagger$ is the ring of convergent power series with coefficients in  $\bfw{C}_p$.  We regard $\nami{\Lambda}_\ratnum$ (resp.\, $\Lambda^\dagger$) as $\nami{\Lambda}_F$ (resp.\, $\Lambda$)-algebras via the homomorphism of $\ywrk{O}$-algebras induced by the following group homomorphisms:
\begin{eqnarray*}
 Z_{F,0}\ni x& \longmapsto &\norm{F_p/\ratnum_p}{x} \in \ratint_p \\
(\mbox{resp.\, }1+\ratint_p \ni x & \longmapsto& x^{T-2} \in \Lambda^\dagger). 
\end{eqnarray*}
\begin{defn}
For an $\intring{F_\spp}$-lattice $L \subset F_\spp^2$, the {\em primitive part} of $L$ is $L\setminus \doi{p}L$ and we denote it by $L'$
\end{defn}
We define several spaces:
\begin{eqnarray*}
X &\defeq &\posuni\big{\backslash}(\intring{F_p}^2)',  \\ X'&\defeq &\posuni\big{\backslash}\lt(\intring{F_\spp}^\times\rt)\times \doi{p} \intring{F_\doi{p}}, \\ \ywrk{W}&\defeq&\posuni\bbs\left(F_\doi{p}^2-\{(0,0)\}\right).  
\end{eqnarray*}
We define the spaces of compactly supported measures on them
\begin{eqnarray*}
\ywrk{D}_* &\defeq& \left\{\mbox{compactly supported measures on }X\right\}, \\
\ywrk{D}_*' &\defeq& \left\{\mbox{compactly supported measures on }X'\right\}, \\
\ywrk{D} &\defeq& \left\{\mbox{compactly supported measures on }\ywrk{W}\right\}. \\
\end{eqnarray*}
Via the zero extension, we have natural inclusions:
\[\ywrk{D}_*'\subset\ywrk{D}_*\subset\ywrk{D}\]
For any  function $f$ on $\ywrk{W}$, we define an action of $\gamma=\gyoretu{a}{b}{c}{d}\in\rmf{GL}_2(F_\spp)$ by
\[f|\gamma(x,y)\defeq f\big(ax+by,cx+dy\big).\]
Then we define the action of $\rmf{GL}_2(F_\spp)$ on $\ywrk{D}$ by
\[ \int_S f \ d(g\cdot \mu) \defeq \int_{g^{-1}(S)} f|g \ d\mu. \]
where  $\mu \in \ywrk{D}$, $g\in \rmf{GL}_2(F_p)$ and $S$ a compact subset of $\ywrk{W}$.
\begin{defn}
Let $\rmf{pr}:\ywrk{D}\rightarrow\ywrk{D}_*$ (resp.\, $\rmf{pr}':\ywrk{D}\rightarrow\ywrk{D}_*'$) be a natural projection via restrictions.  We define the action of $g\in\Delta_0(\doi{n}^+, \doi{n}^-)$ (resp.\, $\Delta_0(\spp\doi{n}^+, \doi{n}^-)$) on $\mu\in\ywrk{D}_*$(resp.\, $\ywrk{D}_*'$) by
\begin{eqnarray*}
(g,\mu) & \longmapsto & \rmf{pr}(g\cdot\mu)\\
(\mbox{resp.\, } (g,\mu) & \longmapsto & \rmf{pr}'(g\cdot\mu)).
\end{eqnarray*}
\end{defn}	
		
\begin{defn}
The set of weight characters $\ywrk{X}_F$ is
\[\ywrk{X}_F \defeq \rmf{Hom}_{\rmf{cont}}\left(Z_{F,0}, \cpxnum_p\right).\] 

\end{defn}
\begin{defn}
A function $f$ on $X$ is said to be homogeneous with respect to the weight character  $\phi\in\ywrk{X}_F$ if
\[ f\bigg|\gyoretu{c}{0}{0}{c}=\phi(\uesen{c})f \]
for any $c\in \intring{F_p}^\times$ and $\uesen{c}$ is an image in $Z_{F,0}$.
\end{defn}
\begin{defn}
For any $k\ge 2$, We define a weight character $P_k$ by $\swtcha{}{k-2}$ .
\end{defn}

\begin{defn}
We put
\begin{eqnarray*}
\ywrk{D}^{\rmf{cycl}}& \defeq& \ywrk{D} \otimes_{\nami{\Lambda}_F}\nami{\Lambda}_\ratnum  \\
\ywrk{D}^{\rmf{cycl}, \dagger} &\defeq& \ywrk{D}^{\rmf{cycl}} \otimes_{\cpxnum_p\yama{\otimes}\Lambda} \Lambda^{\dagger}. \\
\end{eqnarray*}
where $\cpxnum_p\yama{\otimes}\Lambda:=\cpxnum_p[[1+p\ratint_p]]$.
Similarly we define $\ywrk{D}_*^{\rmf{cycl}}$, $\ywrk{D}_*^{\rmf{cycl}, \dagger}$, $(\ywrk{D}_*')^{\rmf{cycl}}$ and $(\ywrk{D}_*')^{\rmf{cycl}, \dagger}$.
\end{defn}
\begin{rmk}
For $\mu\in\ywrk{D}^{\rmf{cycl}}$, we can consider the integration $\ds\int f d\mu$ for only homogeneous functions: Let $\mu = \sum_{i=1}^r \mu_i\otimes_{\Lambda_F}\lambda_i $ with $\mu \in \ywrk{D}$ and  $\lambda_i \in \nami{\Lambda}_\ratnum$. Then for any homogeneous function with the weight character $\swtcha{}{s-2}$, we define
\[ \int f d\mu \defeq \sum_{i=1}^r \langle\lambda_i(s)\rangle\int f d\mu_i, \]
where $s\mapsto \langle\lambda(s)\rangle$ is induced by the composition $\ratint_p^\times \rightarrow 1+p\ratint_p \rightarrow \Lambda^\dagger$. Note that the definition is well-defined.  Similarly, we can define an integral of an element of $\ywrk{D}^{\rmf{cycl}, \dagger}$ when the $s$ is sufficiently close to $2$.
\end{rmk}
\begin{defn}
For any $k\ge2$ and for any $\eta \in \ywrk{B}_{k}$, we define the function  $\nami{\eta}_{}$ on $\ywrk{W}$
\[\nami{\eta}_{}(x,y)\defeq 
\omega_F(x)^{2-k}\eta(x,y), \]
where $\omega_F(x) \defeq \chi_{F,\rmf{cycl}}(x)\big/\lt\langle\chi_{F,\rmf{cycl}}(x)\rt\rangle$.
\end{defn}

\begin{rmk}
This $\nami{\eta}$ is homogeneous with the weight character $P_{n,v}$.
\end{rmk}

\begin{defn}
\label{specialization}
For any $k\ge 2$, we define the specialization map to weight $k$ by
\[
\xymatrix{
\rho_{k}:\hspace{-12mm}&\ \ywrk{D}_* \ar[r] & V_k \\
&\mu \ar@{}[u]|{\mws{$\in$}} \ar@{|->}[r] & \left[ \eta \mapsto \displaystyle\int_{X'} \nami{\eta}(x,y) d\mu(x,y)  \right] \ar@{}[u]|{\mws{$\in$}}.
}
\]
The map $\rho_{k}$ is $\Delta_1(\spp,\doi{n}^-)$-equivariant.  By the same formula, we  define  the specialization map $\rho_{k}$ on $\ywrk{D}_*^\rmf{cycl}$ and on  $\ywrk{D}_*^{\rmf{cycl}, \dagger}$ if $k$ is sufficiently close to $2$ $p$-adically.
\end{defn}
\begin{rmk}
\label{specializationnormk}
If $k\equiv2 \mbox{ \rmf{mod} }p-1$, the specialization map $\rho_{k}$ is $\Delta_0(\spp, \doi{n}^-)$-equivariant. More precisely, we have
\[\rho_{k}(u\cdot\mu)=\omega_F^{2-k}(u)u\cdot\rho_{k}(\mu)\]
for any $\mu \in \ywrk{D}_*$ and $u\in \Sigma_0(\spp,\doi{n}^-)$.
\end{rmk}
This specialization map induces 
\[(\rho_{k})_*:S\left(\Sigma_0(\doi{n}^+,\doi{n}^-), \ywrk{D}_*\right) \longrightarrow S\left(\Sigma_1\left(\spp\doi{n}^+, \doi{n}^-\right), V_{k}\right) \]
(As in Remark \ref{specializationnormk}, we can replace the $\Sigma_1$ of right hand side with $\Sigma_0$ if $k\equiv2 \mbox{ \rmf{mod} }p-1$).

Similarly we also define specialization maps $(\rho_{n, v})_*$ on the spaces $S\left(\Sigma(\doi{n}^+,\doi{n}^-), \ywrk{D}_*^\rmf{cycl} \right)$ and on $S\left(\Sigma(\doi{n}^+,\doi{n}^-), \ywrk{D}_*^{\rmf{cycl}, \dagger} \right) $ if $k$ close to 2 $p$-adically (Note that since the set of double coset $B^\times\bbs\yama{B}^\times\big/\Sigma(\doi{n}^+, \doi{n}^-)$ is finite, the specialization map is defined for all $k\ge2$ such that $k$ is sufficiently close to $2$). 

By definition, the specialization map commutes with the action of Hecke operators (in Definition \ref{Hecke operators})
\[\rho_{k}\circ T(\doi{a})=T(\doi{a})\circ\rho_{k}\]
\[ \rho_{k}\circ T(\doi{b},\doi{b})=T(\doi{b},\doi{b})\circ\rho_{k} \]
for all nonzero ideals $\doi{a}$ prime to $\spp\doi{n}^-$ and $\doi{b}$ prime to $\spp\doi{n}^+\doi{n}^-$. 

On the other hand, the action of $T(\spp)$ on $S\left(\Sigma(\doi{n}^+, \doi{n}^-), \ywrk{D}_*\right)$ is also transferred to the action of $T(\spp)$ on $S\left(\Sigma(\doi{n}^+, \doi{n}^-), V_{k}\right)$:
\begin{prop}
For any $k\ge2$, we have
\[\rho_{k}\circ T(\spp) =T(\spp)\circ\rho_{k}. \]
\end{prop}

\begin{proof}
This follows from a simple computation by using the formulae (\ref{migikoset1}), (\ref{migikoset2}) of Definition \ref{Hecke operators}.
\qed\end{proof}

\subsection{Hida deformation of measure valued forms}
\label{subsechidadeformation}
In the  previous section, we have defined the specialization maps $\rho_{k}$.  These maps give a $p$-adic family of quaternionic automorphic forms.  By Hida's theory, there exists a $p$-adic family of Hecke eigenforms.

Let $\bfb{T}$ be the free polynomial algebra over $\ratint$ in the symbols $\big\{T(\doi{a})\big\}$ for ideals $\doi{a}$ such that $(\doi{a},\doi{n}^-)=1$ and $\big\{T(\doi{b},\doi{b})\big\}$ for ideals $\doi{b}$ such that $(\doi{b},\spp\doi{n}^+\doi{n}^-)=1$.  Then $\bfb{T}$ acts on the spaces $S\lt(\Sigma_0\lt(\doi{n}^+, \doi{n}^-\rt), \ywrk{D}_*^{\rmf{cycl}, \dagger}\rt)$ and $S\lt(\Sigma_1\lt(\spp\doi{n}^+, \doi{n}^-\rt), V_{k}\rt)$ as Hecke operators.

We call $\Phi_{k} \in S\left(\Sigma_1(\spp\doi{n}^+, \doi{n}^-), V_{k}\right)$ a {\em Hecke  eigenform} if it is an eigenvector for the action of $\bfb{T}$.  When $\Phi_k$ is a Hecke eigenform, we denote by $a(\doi{a}, \Phi_{k})$ the eigenvalue for $T(\doi{a})$ for an ideal $\doi{a}$ is prime to $\doi{n}^+$.  We call $\Phi_{k}$ is {\em $\spp$-ordinary} if $a(\spp, \Phi_{})$ is a $p$-adic unit.

Similarly, we call $\Phi_\infty \in S\left(\Sigma_0(\doi{n}^+, \doi{n}^-), \ywrk{D}_*^{\rmf{cycl}, \dagger}\right)$ a {\em Hecke eigenform} if it is an eigenvector for the action of $\bfb{T}\otimes_\ratint \Lambda^\dagger$, and we denote by  $a(\doi{a}, \Phi_\infty)(T)\in\Lambda^\dagger$ the eigenvalue of $T(\doi{a})$ of $\Phi_\infty$.  We can show that there exists a positive radius such that $a(\doi{a}, \Phi_\infty)(T)$ can be defined for all $\doi{a}$.

Now we state Hida's theory of lifting a Hecke eigenform to a $p$-adic family, in the style of Greenberg-Stevens in \cite{GS1}.

\begin{thm}
\label{Hidatheory}
Let $\Phi=\Phi_{2} \in S\left(\Sigma_0(\spp\doi{n}^+, \doi{n}^-), V_{2}\right)$ be a $\spp$-ordinary Hecke eigenform, and new at primes dividing $\doi{p}\doi{n}^+$. Then there exists a Hecke eigenform \[\Phi_\infty \in S\left(\Sigma_0(\doi{n}^+, \doi{n}^-), \ywrk{D}_*^{\rmf{cycl}, \dagger}\right)\] such that $\rho_{2}(\Phi_\infty)=\Phi_{2}$.
\end{thm}

The following natural map:
\begin{eqnarray}\label{hokeikeisikinoiso}
S\lt(\Sigma_0\lt(\doi{n}^+,\doi{n}^-\rt), \ywrk{D}_*\rt) \overset{\cong}{\longrightarrow} S\lt(\Sigma_0\lt(\spp\doi{n}^+,\doi{n}^-\rt), \ywrk{D}_*'\rt)\end{eqnarray}
is an isomorphism (see \cite{Mo1}, Appendix I\hspace{-.1em}I). Moreover this isomorphism commutes with the action of Hecke operators $T(\doi{a})$ ($\doi{a}$ is a non-zero ideal relatively prime to $\doi{n}^-$) and $T(\doi{b},\doi{b})$ ($\doi{b}$ is a non-zero ideal relatively prime to $\spp\doi{n}^+\doi{n}^-$).
This isomorphism induces an isomorphism
\[S\lt(\Sigma_0\lt(\doi{n}^+,\doi{n}^-\rt), \ywrk{D}_*^{\rmf{cycl}, \dagger}\rt) \overset{\cong}{\longrightarrow} S\lt(\Sigma_0\lt(\spp\doi{n}^+,\doi{n}^-\rt), (\ywrk{D}_*')^{\rmf{cycl}, \dagger}\rt).\]
We denote by $\uesen{\Phi}_\infty$ the image of $\Phi_\infty$ in Theorem \ref{Hidatheory} via the above isomorphism.

\subsection{The indefinite integrals}
\label{indefinite integral}
As in Theorem \ref{Hidatheory}, let  \[\Phi\in S\lt(\Sigma_0(\spp\doi{n}^+,\doi{n}^-),V_{2}\rt)\] be an automorphic form of weight $2$, and let \[\Phi_\infty\in S\left(\Sigma_0(\doi{n}^+, \doi{n}^-), \ywrk{D}_*^{\rmf{cycl}, \dagger}\right)\] be a measure valued form. The forms $\Phi$ and $\Phi_\infty$ are Hecke eigenforms. Let  $(\phi^1, \dots ,\phi^h)$ and $(\phi_\infty^1, \dots ,\phi_\infty^h)$ be $h$-tuples corresponding to $\Phi$ and  $\Phi_\infty$ respectively. We denote by $\alpha_{\spp}$ and $\alpha_{\spp}(s)$ the eigenvalue of $T(\doi{p})$ on $\Phi$ and $\Phi_\infty$ respectively (Note that $\alpha_{\spp}(0)=\alpha_{\spp}$).

Let $\ratnum(\Phi)$ be the Hecke field of the automorphic form $\Phi$, which is a number field generated over $\ratnum$ by all of the Hecke eigenvalues of $\Phi$.  We can regard $\Phi$ as a $\cpxnum_p$-valued function on a finite set.  Recall that we have fixed an embedding $\uesen{\ratnum}\hookrightarrow\cpxnum_p$

Let $\uesen{\Phi}_\infty \in S\left(\Sigma_0(\spp\doi{n}^+, \doi{n}^-), (\ywrk{D}_*')^{\rmf{cycl}, \dagger}\right)$ be the image of $\Phi$ by the isomorphism as in the end of Section  $\ref{subsechidadeformation}$ (which is also a Hecke eigenform) and let $(\uesen{\phi}_\infty^1, \dots ,\uesen{\phi}_\infty^h)$ be the $h$-tuple corresponding to $\uesen{\Phi}_\infty$. We regard $\Phi_\infty$ and $\uesen{\Phi}_\infty$ as $\ywrk{D}^{\rmf{cycl}, \dagger}$-valued functions by the canonical injection $\ywrk{D}_*^{\rmf{cycl}, \dagger}\hookrightarrow\ywrk{D}^{\rmf{cycl}, \dagger}$ and $\lt(\ywrk{D}_*'\rt)^{\rmf{cycl}, \dagger}\hookrightarrow\ywrk{D}_*^{\rmf{cycl}, \dagger}$via the zero extension. Let $U$ be a $p$-adic neighborhood of $2$ such that for $k\ge2$ with $k \in U$ the weight $k$ specialization $\Phi_{k}$ of $\Phi_\infty$ is defined. 

Fix a point $\tau\in\hfpl{\uesen{\ratnum}_p}$ in the $p$-adic upper half plane.  Let $K/F_\spp$ be a finite extension over $F_\spp$ containing  the element $\tau$. 
Then we define a function $F_s^\tau(x,y)$ on $\ywrk{W}$:
\[F_s^\tau(x,y)\defeq\lt\langle \norm{K/\ratnum_p}{x_\doi{p}-\tau_\doi{p}y_\doi{p}}\rt\rangle^{(s-2)/[K:F_\doi{p}]}. \] 
Note that this function is defined on an open compact subset of $\ywrk{W}$ if $s$ is sufficiently close to $2$. The function $F_s^\tau$ does not depend on the choices of $K$.

\begin{defn}
Let $\tau\in \hfpl{\uesen{\ratnum}_p}$.
We define the functions $\theta_{\phi^i}^\tau$ and $\uesen{\theta}_{\phi^i}^\tau$ as follows:
\begin{eqnarray*}
\theta_{\phi^i}^\tau(s;L) &\defeq&  \alpha_{\spp}(s)^{-\ord_p\lt(\rmf{det}\lt(\xi_{\doi{n}^+}^{-1}(L)\rt)\rt)}c_{\phi^i}(L,L)(F_s^\tau) \mbox{ for } L \mbox{ such that }(L,L) \in \ywrk{P}(\doi{n}^+), \\
\thetab_{\phi^i}^\tau(s;L_1,L_2)&\defeq& \alpha_{\spp}(s)^{-\ord_p\lt({\rm det}\lt(\xi_{\spp\doi{n}^+}^{-1}(L_1,L_2)\rt)\rt)}c_{\uesen{\phi}^i}(L_1,L_2)(F_s^\tau) \mbox{ for }(L_1,L_2) \in \ywrk{P}(\spp\doi{n}^+), \\
\end{eqnarray*}
where $c_{\phi^i}$, $c_{\uesen{\phi}^i}$, $\xi$ and $\ywrk{P}$ are as in Section \ref{another description}.
These functions are defined if $s$ is sufficiently close to 2 p-adically.
\end{defn}
The functions $\theta_{\phi^i}^\tau(s;\cdot)$ and $\thetab_{\phi^i}^\tau(s;\cdot)$ are analytic in the variable $s$. So we can consider the derivative with respect to $s$ and define new functions on $\ywrk{P}(\doi{n}^+)$, $\ywrk{P}(\spp\doi{n}^+)$ as follows:
\begin{defn}
Let $\tau\in \hfpl{\uesen{\ratnum}_p}$.  We define functions $I^\tau$ and $\Ib^\tau$ as follows:
\begin{eqnarray*}
I_{\phi^i}^\tau(L) &\defeq& \lt.\frac{d\theta_{\phi^i}^\tau(s;L)}{ds}\rt|_{s=2} \mbox{ for }\ L \mbox{ such that }  (L,L) \in \ywrk{P}(\doi{n}^+)\\
\Ib_{\phi^i}^\tau(L_1,L_2) &\defeq& \lt.\frac{d\thetab_{\phi^i}^\tau(L_1,L_2)}{ds}\rt|_{s=2} \mbox{ for }\ (L_1,L_2)\in\ywrk{P}(\spp\doi{n}^+)\\
\end{eqnarray*}
\end{defn}

\begin{lem}
\label{thetaiotanosiki}
For $\theta_{\phi^i}^\tau$, $\thetab_{\phi^i}^\tau$, $I_{\phi^i}^\tau$ and $\Ib_{\phi^i}^\tau$,  we have the following formulae:
\begin{eqnarray}
\theta_{\phi^i}^\tau(s;L) &=& \sum_{\lt\{L'\relmiddle|(L,L')\in\ywrk{P}(\spp\doi{n}^+)\rt\}}\thetab_{\phi^i}^\tau(s;L,L') \label{thetanosiki1}\\
&=&\thetab_{\phi^i}^\tau(s;L,L') +\thetab_{\phi^i}^\tau\lt(s;\frac{1}{p}L',L\rt) \ \mbox{ for }\ (L,L')\in \ywrk{P}(\spp\doi{n}^+)\label{thetanosiki2}\\
\theta_{\phi^i}^\tau(s;pL) &=& \alpha_{\spp}(s)^{-2}\theta_{\phi^i}^\tau(s;L) \label{thetanosiki3}\\
\thetab_{\phi^i}^\tau(s;pL_1,pL_2) &=& \alpha_{\spp}(s)^{-2}\thetab_{\phi^i}^\tau(s;L_1,L_2) \label{thetanosiki4}
\end{eqnarray}
\begin{eqnarray}
I_{\phi^i}^\tau(L) &=& \sum_{\lt\{L'\relmiddle|(L,L')\in\ywrk{P}(\spp\doi{n}^+)\rt\}}\Ib_{\phi^i}^\tau(L,L') \label{iotanosiki1}\\
&=& \Ib_{\phi^i}^\tau(L,L') +\Ib_{\phi^i}^\tau\lt(\frac{1}{p}L',L\rt) \ \mbox{ for }\ (L,L')\in \ywrk{P}(\spp\doi{n}^+)\label{iotanosiki2}\\
I_{\phi^i}^\tau(pL) &=&\alpha_{\spp}^{-2}I^\tau(L) \label{iotanosiki3}\\
\Ib_{\phi^i}^\tau(pL_1,pL_2) &=& -2\alpha_{\spp}'(0)\alpha_{\spp}^{-3}\thetab_{\phi^i}^\tau(0;L_1,L_2)+\alpha_{\spp}^{-2}\Ib_{\phi^i}^\tau(L_1,L_2) \label{iotanosiki4}
\end{eqnarray}
\end{lem}
\begin{proof}
The equalities (\ref{thetanosiki3}), (\ref{thetanosiki4}) are easily proved by simple computations. For the equalities (\ref{thetanosiki1}) and (\ref{thetanosiki2}),  take  $g\in\rmf{GL}_2(F_\spp)$ satisfying $g(\intring{F_\spp}\times\intring{F_\spp})=L$. Then we have
\begin{eqnarray*}
c_{\phi^i}(L,L)\lt({F_s^\tau}\rt)&=&\int_X {F_s^\tau}|g\  d\phi_\infty^i(g)\\
&=&\int_{\spp\intring{F_\spp}\times\intring{F_\spp}^\times} {F_s^\tau}|g\ d\phi_\infty^i(g) + \int_{\intring{F_\spp}^\times \times \intring{F_\spp}} {F_s^\tau}|g\ d\phi_\infty^i(g).
\end{eqnarray*}
For the first term, let $L':=g(\spp\intring{F_\spp}\times\intring{F_\spp})$. Then by definition
\begin{eqnarray}
\int_{\spp\intring{F_\spp}\times\intring{F_\spp}^\times} {F_s^\tau}|g\ d\phi_\infty^i(g) = c_{\uesen{\phi}^i}(L,L')\lt({F_s^\tau}\rt). \label{siki0}
\end{eqnarray}
For the second term, let $\sigma_\infty:=\gyoretu{1}{0}{0}{\pi_{\doi{p}}}$ and for any $c\in\intring{F_\spp}$, let $\sigma_c=\gyoretu{1}{0}{\pi_{\doi{p}} c}{\pi_\doi{p}}$ ($\pi_\doi{p}$ is a uniformizer in $\intring{F_\spp}$). Then we have
\begin{eqnarray}
\int_{\intring{F_\spp}^\times \times \intring{F_\spp}} {F_s^\tau}|g\ d\phi_\infty^i(g)  &=& 
\int_{\intring{F_\spp}^\times \times \spp\intring{F_\spp}} {F_s^\tau}|g\sigma_\infty^{-1}\ (\sigma_\infty\cdot d\phi_\infty^i)(g) \nonumber\\
&=& \sum_{c\in \intring{F_\spp}/\spp}\int_{\sigma_c\lt({\intring{F_\spp}^\times \times \spp\intring{F_\spp}}\rt)} {F_s^\tau}|g\sigma_\infty^{-1}\ (\sigma_\infty\cdot d\phi_\infty^i)(g) \nonumber\\
&=& \sum_{c\in \intring{F_\spp}/\spp}\int_{{\intring{F_\spp}^\times \times \spp\intring{F_\spp}}} {F_s^\tau}|g\sigma_\infty^{-1}\sigma_c\ d\uesen{\phi}_\infty^i(g\sigma_\infty^{-1}\sigma_c) \label{siki1}\\
&=&\lt(T(\spp)\uesen{\phi}_\infty^i\rt)(g\sigma_\infty^{-1})\lt({F_s^\tau}\rt) \nonumber \\
&=&\alpha_{\spp}(s)c_{\uesen{\phi}_\infty^i}\lt(\frac{1}{p}L',L\rt)\lt({F_s^\tau}\rt). \label{siki2}
\end{eqnarray}
The equalities (\ref{siki0}), (\ref{siki1}) imply (\ref{thetanosiki1}), and the equalities  (\ref{siki0}), (\ref{siki2}) imply (\ref{thetanosiki2}).

The equalities (\ref{iotanosiki1}), (\ref{iotanosiki2}) and (\ref{iotanosiki4}) easily follows by differentiating (\ref{thetanosiki1}), (\ref{thetanosiki2}) and (\ref{thetanosiki4}) respectively. For (\ref{iotanosiki3}), since $\Phi$ is a newform, $\theta_{\phi^i}^\tau\big|_{s=2}=0$ by the formula (\ref{thetanosiki1}), it also follows by differentiating the equality (\ref{thetanosiki3}).
\qed\end{proof}

From now on, we assume the following condition:
\begin{ass}
\label{katei}
$\alpha_{\spp}^2=1$
\end{ass}
\begin{rmk}
The above assumption holds if $\Phi_{2}$ comes from a modular elliptic curve over $F$ with multiplicative reduction at $\spp$.
\end{rmk}
As in the proof of Lemma \ref{thetaiotanosiki}, we have $\theta_{\phi^i}^\tau\big|_{s=2}=0$. Thus by the assumption (\ref{katei}) and the equality (\ref{thetanosiki1}) , we regard $\thetab_{\phi^i}^\tau\big|_{s=2}$ as an element in $\meas{\prjl{F_\doi{p}}}{\intring{\ratnum(\Phi)}}$. By definition, $\thetab_{\phi^i}^\tau\big|_{s=2}$ does not depend on the choice of $\tau$. So we denote it by $\mu_{\phi^i}$.  Then the measure $\mu_{\phi^i} $ is the same as one defined in \cite{BD1} Section 2.9, namely
\begin{eqnarray}\mu_{\phi^i} = \rmf{proj}_*\lt(c_{\phi^i}(\intring{F_p}^2, \intring{F_p}^2)\rt), \label{phinomeasure}\end{eqnarray}
where $\rmf{proj}$ is the following map:

\[\mapdiag{\rmf{proj}:}{\ywrk{W}}{\prjl{F_p}}{(x,y)}{\ds\frac{x}{y}.}\]

The function $I_{\phi^i}^\tau$ is a function on the set of vertices in the Bruhat-Tits tree $\ywrk{T}_{F_\spp}$ since we have $I_{\phi^i}^\tau(L)=I_{\phi^i}^\tau(pL)$ by Assumption \ref{katei} (see the formula (\ref{iotanosiki3})).  We regard $I_{\phi^i}^\tau$ as an element in $\rmf{Hom}_\ratint\lt(\rmf{Div}\lt(\hfpl{\cpxnum_p}\rt)_, \cpxnum_p\rt)$ as in the end of Section \ref{psinmeasure}.
\begin{defn}
For any $x \in \uesen{\ratnum}_p$, we define
\[\logNorm_p(x)\defeq\frac{1}{[L:\ratnum_p]}\log_p\norm{L/\ratnum_p}{x},\]
where $L$ is a finite extension over $\ratnum_p$ containing $x$. Recall that $\log_p:\ratnum_p^\times\rightarrow\ratnum_p$ is the Iwasawa logarithm map satisfying $\log_p(p)=0$.
\end{defn}
Since $\Phi$ is a function on a finite set, the image of $\mu_{{\phi}^i}$ is a finitely generated abelian subgroup of $\cpxnum_p$. Thus as in Section \ref{batusekibun}, we consider the multiplicative integral attached to $\mu_{\phi^i}$. Then we have the following two lemmas:
\begin{lem}
\label{ordtobatu}
For $x,y \in \hfpl{\cpxnum_p}$ and $\tau \in \hfpl{\uesen{\ratnum}_p}$, we have
\begin{eqnarray}
I_{\phi^i}^\tau(x)-I_{\phi^i}^\tau(y) &= &  2\alpha_{\spp}\alpha_{\spp}'(0)\mu_{\phi^i}([x]-[y]) \label{ordtobatu1}\\
&=& 2\alpha_{\spp}\alpha_{\spp}'(0)\cdot\iota\lt\{\ord_p\otimes_\ratint \rmf{id}_{\intring{\ratnum(\Phi)}}\lt(\batuint{[x]-[y]}{\omega_{\mu_{\phi^i}}}\rt)\rt\}, \label{ordtobatu2}
\end{eqnarray}
where $\iota:\cpxnum_p\otimes_\ratint\cpxnum_p\rightarrow\cpxnum_p$ is a natural multiplication map.
\end{lem}
\begin{proof}
The equality (\ref{ordtobatu2}) follows from Proposition \ref{ordtobatuint}.  For the equality (\ref{ordtobatu1}), we may assume $\red_{F_\spp}(x)$ and $\red_{F_\spp}(y)$ are vertices of $\ywrk{T}_{F_\spp}$.  Then it follows from the formulae (\ref{iotanosiki2}) and (\ref{iotanosiki4}).
\qed\end{proof}
\begin{lem}
\label{logtobatu}
For  $x \in \hfpl{\cpxnum_p}$ and  $\tau_1, \tau_2 \in \hfpl{\uesen{\ratnum}_p}$, we have
\begin{eqnarray}
I_{\phi^i}^{\tau_1}(x)-I_{\phi^i}^{\tau_2}(x) &=&\int_{\prjl{F_\doi{p}}}\logNorm_p\lt(\frac{t-\tau_1}{t-\tau_2}\rt) d\mu_{\phi^i}(t) \label{logtobatu1}\\
&=&\iota\lt\{\logNorm_p\otimes_\ratint \rmf{id_{\intring{\ratnum(\Phi)}}}\lt(\batuint{[\tau_1]-[\tau_2]}{\omega_{\mu_{\phi^i}}}\rt)\rt\} \label{logtobatu2}
\end{eqnarray}
where $\iota:\cpxnum_p\otimes_\ratint\cpxnum_p\rightarrow\cpxnum_p$ is a natural multiplication map.
\end{lem}
\begin{proof}
The equality (\ref{logtobatu2}) follows by definition.  For (\ref{logtobatu1}), by Lemma \ref{ordtobatu}, the difference $I^{\tau_1}-I^{\tau_2}$ is a constant function. Thus we may assume $\red_{F_\spp}(x)$ is the class of lattice $\intring{F_\spp}^2$.  Then  by definition, 
\begin{eqnarray*}
I_{\phi^i}^{\tau_1}(x)-I_{\phi^i}^{\tau_2}(x)&=&\lt.\frac{d}{ds}\phi_\infty(1)(F_s^{\tau_1}-F_s^{\tau_2})\rt|_{s=2} \\
&=&\int_X \logNorm_p\lt(\frac{x-\tau_1y}{x-\tau_2y}\rt) d\phi_\infty(1) \\
&=&\int_{\prjl{F_\spp}} \logNorm_p\lt(\frac{t-\tau_1}{t-\tau_2}\rt)\ d\mu_{\phi^i}(t).
\end{eqnarray*}
The last equality follows from (\ref{phinomeasure}).
\qed\end{proof}

\begin{defn}
We define the {\em indefinite integral} $I_{\phi^i}$ attached to $\phi^i$ as follows:
\[\mapdiag{I_{\phi^i}:}{\hfpl{\uesen{\ratnum}_p}}{\cpxnum_p}{\tau}{I_{\phi^i}(\tau)\defeq I_{\phi^i}^\tau(\tau)}\]
\end{defn} 
\begin{prop}
\label{Gamma actions of the indefinite integral}
The discrete subgroup $\nami{\Gamma}_i \subset B^\times$ acts on the indefinite integral $I_{\phi^i}$ as follows:
\[I_{\phi^i}(\gamma \tau)=\alpha_{\spp}^{-\ord_\spp\lt({\rmf{det}(\gamma)}\rt)}I_{\phi^i}(\tau)\]
for any $\gamma \in \nami{\Gamma}_i$. In particular,  the indefinite integrals are $\Gamma_i$-equivariant. 
\end{prop}
\begin{proof}
We have the following formula
\[I_{\phi^i}^\tau(\gamma x)=\alpha_{\spp}^{-\ord_\spp\lt({\rmf{det}(\gamma)}\rt)}I_{\phi^i}^{\gamma^{-1}\tau}(x),\]
by a simple computation. The assertion of Proposition \ref{Gamma actions of the indefinite integral} follows from this by putting $x=\gamma^{-1}\tau$.
\qed
\end{proof}
\begin{defn}
Let $\Phi$ be a quaternionic automorphic form satisfying the conditions as in the beginning of this section.  We define {\em the indefinite integral $I_\Phi$ attached to $\Phi$}  by
\[\mapdiag{I_\Phi:}{\bigoplus_{i=1}^h\rmf{Div}\lt(\Gamma_i\bbs\hfpl{\uesen{\ratnum}_p}\rt)}{\cpxnum_p}{x:=(x_i)_{i=1}^h}{I_\Phi(x)\defeq\ds\sum_{i=1}^h I_{\phi^i}(x_i)}.\]
\end{defn}
\begin{prop}
\label{Hecke actions of indefinite integrals}
Let $\Phi$ be a quaternionic automorphic form satisfying the conditions as in the beginning of this section. The indefinite integral attached to $\Phi$ respects the action of Hecke operators. Namely, for all non-zero ideal $\doi{a}$, we have
\[I_\Phi\lt(T(\doi{a})z\rt)=\alpha(\doi{a}, \Phi)I_\Phi\lt(z\rt).\]
\end{prop}
\begin{proof}
Let $\doi{a}$ be an ideal in $\intring{F}$ prime to $\doi{n}^-$. At first, we explicitly describe the action of $T(\doi{a})$.
As in Definition \ref{Hecke operators}, we set 
\[\left\{ x \in \Delta_0\lt(\doi{n}^+, \doi{n}^-\rt) \relmiddle| \rmf{Nrd}_{B/F}(x)\intring{F}=\doi{a} \right\}=\bigsqcup_m \sigma_m \Sigma_0\lt(\doi{n}^+, \doi{n}^-\rt),\]
and as in Section \ref{another description}, we set
\[\yama{B}^\times = \bigsqcup_{i=1}^h B^\times x_i B_\doi{p}^\times\Sigma_0\lt(\doi{n}^+,\doi{n}^-\rt).\] 
For any $x_i$ and $\sigma_m$, there exist elements $b_{i,m}\in B^\times$, $g_{i,m}\in \rmf{GL}_2(F_\spp)$, $g_{i,m}'\in\Sigma_0\lt(\doi{n}^+,\doi{n}^-\rt)^\spp$ and a number  $1\le l_{i,m}\le h$ such that \[x_i \sigma_m=b_{i,m}x_{l_{i,m}}g_{i,m}g_{i,m}'.\]
We note that there exists $\gamma_i\in\nami{\Gamma}_i$ such that $\ord_\spp\lt(\rmf{det}(\gamma_i)\rt)=1$ (see \cite{Vig}, Corollary 5.9).  Then the actions of $T(\doi{a})$ on an automorphic form $\phi^i$ and an element $(z_i)\in\ds\bigoplus_{i=1}^h\rmf{Div}\lt(\Gamma_i\bbs\hfpl{\uesen{\ratnum}_p}\rt)$ are given as follows:
\begin{eqnarray}
\lt(T(\doi{a})\phi^i\rt)(g)&=&\sum_m\phi^{l_{i,m}}(g_{i,m}g), \label{hecke}\\
T(\doi{a})z_i&=&\sum_m \lt(\gamma_{l_{i,m}}^{-v_{i,m}}g_{i,m}z_i\rt)_{l_{i,m}} \nonumber,
\end{eqnarray}
where $v_{i,m}\defeq\ord_\spp\lt(\rmf{det}(g_{i,m})\rt)$.

Now we get back to the proof of Proposition \ref{Hecke actions of indefinite integrals}. We put  \[z=(z_i)_{i=1}^h\in\ds\bigoplus_{i=1}^h\rmf{Div}(\Gamma_i\bbs\hfpl{\uesen{\ratnum}_p}).\] 
Then we may assume $x_i$ is a class of $\tau_i\in\hfpl{\uesen{\ratnum}_p}$ and $\red_{F_\doi{p}}(\tau_i)$ is a class of a lattice $L_i\in\ywrk{P}(\doi{n}^+)$.  Take an element $g\in \rmf{GL}_2(F_\doi{p})$ satisfying $g(\intring{F_\doi{p}}^2)=L_i$.  Then the assertion follows from the following calculation.
\begin{eqnarray*}
I_\Phi\lt(T(\doi{a})z_i\rt)&=&\sum_mI_{\phi^{l_{i,m}}}(\gamma_i^{-v_{i,m}}g_{i,m}z_i) \\
&\underset{\mbox{\scriptsize Prop \ref{Gamma actions of the indefinite integral}}}{=}& \sum_m\alpha_p^{v_{i,m}}\lt.\frac{d}{ds}\alpha_\spp(s)^{-v_{i,m}}\alpha(s)^{-\ord_\doi{p}(\rmf{det}(g))}\int_X F_s^{g_{i,m}g\tau}\big|g_{i,m}g\ d\phi_\infty^{l_{i,m}}(g_{i,m}g)\rt|_{s=2}\\
&=&\lt.\frac{d}{ds}\alpha_\spp(s)^{-\ord_\spp\lt(\rmf{det}(g)\rt)}\int_X F_s^\tau \big|g\ d\sum_m\phi^{l_{i,m}}(g_{i,m}g) \rt|_{s=2}\\
&\underset{(\ref{hecke})}{=}&\alpha_\spp(\doi{a},\Phi)\lt.\frac{d}{ds}\alpha_\spp(s)^{-\ord_\spp\lt(\rmf{det}(g)\rt)}\int_X F_s^\tau \big|g\ d\phi^i(g) \rt|_{s=2}\\
&=&\alpha_\spp(\doi{a},\Phi)I_\Phi(z_i).
\end{eqnarray*}
\qed
\end{proof}
\begin{thm}
\label{indefinite integral no thm}
As in the beginning of Section \ref{indefinite integral}, let $\Phi$ be a quaternionic automorphi form  satisfying Assumption \ref{katei}. Let $(\phi^1,\dots,\phi^h)$ be an $h$-tuple corresponding to $\Phi$. Let $\mu_{\phi^i}$ be the measure attached to $\phi^i$ (see the equation (\ref{phinomeasure})). For any $\tau_1=(\tau_{1,i})_{i=1}^h$ and $\tau_2= (\tau_{2,i})_{i=1}^h \in \ds\prod_{i=1}^h\hfpl{\uesen{\ratnum}_p}$, we have
\[I_{\Phi}(\tau_1)-I_{\Phi}(\tau_2) = \iota\lt\{\bigg(\lt(\logNorm_p+2\alpha_{\spp}\alpha_{\spp}'(0)\cdot\ord_p\rt)
\otimes_\ratint\rmf{id}_{\intring{\ratnum(\Phi)}}\bigg)\lt(\prod_{i=1}^h\batuint{[\tau_{1,i}]-[\tau_{2,i}]}{\omega_{\mu_{\phi^i}}}\rt)\rt\},\]
where $\iota:\cpxnum_p\otimes_\ratint\cpxnum_p\rightarrow\cpxnum_p$ is a natural multiplication map.
\end{thm}
\begin{proof}
The assertion follows from Lemma \ref{ordtobatu} and Lemma \ref{logtobatu}.
\qed\end{proof}

\section{Main results}
\subsection{$p$-adic uniformization of Shimura curves}
\label{p-adic uniformization of Shimura curves}
We use the same notation as in Section \ref{Quaternionic automorphic forms}. Let $F$ be a totally real field and $B$ a definite quaternion algebra over $F$ as in Section \ref{Basic definitions}.  Fix an archimedean place $\infty_0$ of $F$. We denote by $\ywrk{B}$ the definite quaternion algebra over $F$ obtained from $B/F$ by switching the invariants at $\infty_0$ and $\doi{p}$, namely
\[
\begin{array}{rcccl}
\inv_{\infty_0}\ywrk{B}&=&\inv_\doi{p}B &=&0 \\
\inv_{\doi{p}}\ywrk{B}&=&\inv_{\infty_0}B&=&1/2 \\
\inv_{v}\ywrk{B} &=& \inv_{v}B &\multicolumn{2}{l}{\mbox{ for any }v\neq\infty_0, \doi{p}.}\end{array}
\]
We fix an isomorphism $\ywrk{B}\otimes_F(F_{\infty_0})\overset{\cong}{\rightarrow} M_2(\bfw{R})$. Let $\ywrk{R}$ be an Eichler order of level $\doi{n}^+$.  By Shimura's theory, there exists a Shimura curve $X_\ywrk{B}(\doi{n}^+)$, which is a proper smooth curve defined over $F$, whose $\cpxnum$-valued points are given by the double coset:
\[X_\ywrk{B}(\doi{n}^+)(\cpxnum)=\ywrk{B}^\times\bbs\ywrk{B}^\times(\adele_F)\big{/}\yama{\ywrk{R}}^\times\cdot\rmf{SO}_2(\bfw{R}), \]
where $\yama{\ywrk{R}}\defeq\ywrk{R}\otimes_{\intring{F}}\yama{\intring{F}}$.
Let $\ywrk{B}_+^\times$ be the subgroup of $\ywrk{B}^\times$ consisting of elements whose  images by the reduced norm are totally positive elements. We have a double coset decomposition of the following form: \[\ywrk{B}^\times(\adele_{F,f})=\bigsqcup_{i=1}^h \ywrk{B}_+^\times y_i\yama{\ywrk{R}}^\times,\] 
where $h=\#\nnm{Cl}_F^+$ is the narrow class number of $F$.  We define a subgroup  $\Delta_i\subset\ywrk{B}_+^\times$ by
\[\Delta_i \defeq \ywrk{B}_+^\times \cap y_i \yama{\ywrk{R}}^\times y_i^{-1}.\]
The set of $\cpxnum$-valued points of the Shimura curve $X_\ywrk{B}(\doi{n}^+)$ can be written as
\[X_\ywrk{B}(\doi{n}^+)(\cpxnum) \cong \bigsqcup_{i=1}^h \Delta_i \bbs \bfw{H},\]
where $\bfw{H}:=\left\{ z\in \cpxnum \big| \rmf{Im}(z)>0 \right\}$ is the Poincar\'e upper half plane.

We recall the $p$-adic uniformization of the Shimura curve $X_\ywrk{B}(\doi{n}^+)$. By the theorem of  Cerednik-Drinfeld (see \cite{BC} and \cite{BZ}), we have a rigid analytic isomorphism
\begin{eqnarray}
\label{Cerednik-Drinfeld theorem} 
X_\ywrk{B}(\doi{n}^+)(\cpxnum_p) = \bigsqcup_{i=1}^h \Gamma_i \bbs \hfpl{\cpxnum_p},\end{eqnarray}
where $\Gamma_i=\Gamma_i(\doi{n}^+,\doi{n}^-)\subset B_\spp^\times\cong\rmf{GL}_2(F_\spp)$ as in (\ref{gammai}) of Section \ref{another description}.
Let \[X_i=\Gamma_i\bbs\hfpl{\cpxnum_p}\hspace{3mm} \mbox{($i=1,\dots,h$)}\] be the connected components of $X_\ywrk{B}(\doi{n}^+)(\cpxnum_p)$.  

Let $\Phi$ and $h$-tuple $(\phi^1, \dots ,\phi^h)$ be automorphic forms as in the beginning of Section \ref{indefinite integral}. Let $\mu_{\phi^i}$ be the measure attached to $\phi^i$ on $\prjl{F_\spp}$ constructed as in (\ref{phinomeasure}).

\begin{defn}
\label{logNormnodefdayon}
Let $A$ be the abelian variety of $\Phi$-component in $\bigoplus_{i=1}^{h}J(X_i)\otimes_\ratint\ratnum$. Then we define \[\logNorm_p^A: {A(\cpxnum_p)\otimes_\ratint\ratnum}\longrightarrow{\cpxnum_p}\]
by
\[\logNorm_p^A\lt((P_i)_i\rt)\defeq\sum_{i=1}^h\logNorm_p^{X_i}(P_i)(\mu_{\phi^i}). \] 
for any $(P_i)_i\in A(\cpxnum_p)\otimes_\ratint\ratnum\subset \bigoplus_{i=1}^hJ(X_i)\otimes_\ratint\ratnum$.  Here, $\logNorm_p^{X_i}$ is constructed in the end of Section \ref{L-invariants}.
\end{defn}

 Now we state one of the main results of this paper.

\begin{thm}
\label{main results}
Let $\Phi$, $X_i$ and $\phi^i$ be as above. Let $J(X_i)$ be the Jacobian variety of $X_i$. Let $A$ be the abelian variety of $\Phi$-component in $\bigoplus_{i=1}^{h}J(X_i)\otimes_\ratint\ratnum$ . Then, for any $P=(P_i)_{i=1}^h\in A(\cpxnum_p)\otimes_\ratint\ratnum$, we have
\[I_{\Phi}\lt(\sum_{i=1}^h\nami{P}_i\rt)=\logNorm_p^{A}(P),\]
where $\nami{P}_i$ is a lift of $P_i$ in $\rmf{Div}_0(X_i)$.
\end{thm}
\begin{proof}
Let $\mu_i$ be the universal measure associated to $X_i$ as in Theorem \ref{manindrinfeld}.  
We have the decomposition into Hecke eigenspaces:
\[\cpxnum_p\otimes_\ratint\lt(\bigoplus_{i=1}^h\rmf{Coker}(\rmf{Tr})_{\Gamma_i}\rt)=\bigoplus_j V_j.\]
Let $V_0$ be an eigenspace corresponding to $\Phi$.  Put $P\defeq\sum_iP_i$ and $\nami{P}\defeq\sum_i\nami{P}_i$.  Then by Theorem \ref{indefinite integral no thm} and as in Section \ref{L-invariants} , we have:
\[
\begin{aligned}&I_{\phi^i}(\nami{P}_i)\\
&=\iota\lt[\sum_i\lt.\lt\{ \logNorm_p\otimes_\ratint \mu_{\phi^i} +2 \alpha'(0)\lt(\rmf{id}_{\cpxnum_p}\otimes_\ratint \mu_{\phi^i}\rt)\circ \lt(\ord_p \otimes_\ratint \rmf{id}\rt)\rt\}\rt|_{V_0}\lt(\batuint{\nami{P}_i}{\omega_{\oplus_i\mu_i}}\rt)\rt]
\end{aligned}
\]
\[
\begin{aligned} &\logNorm_p^{X_i}(P_i)(\mu_{\phi^i})\\
&=\iota\lt[\sum_i\lt.\lt\{\logNorm_p\otimes_\ratint \mu_{\phi^i} - \lt(\rmf{id}_{\cpxnum_p}\otimes_\ratint \mu_{\phi^i}\rt)\circ \ywrk{L}\circ \lt(\ord_p \otimes_\ratint \rmf{id}\rt)\rt\}\rt|_{V_0}\lt(\batuint{\nami{P}_i}{\omega_{\mu_i}}\rt)\rt],\end{aligned}
\]
where $\iota:\cpxnum_p\otimes_\ratint\cpxnum_p\rightarrow\cpxnum_p$ is the national multiplication map and $\ywrk{L}\defeq\ywrk{L}_{\logNorm_p}$ is the $L$-invariant attached to the homomorphism $\logNorm_p$.  By similar arguments as in Section \ref{L-invariants}, we have
\[\lt.-2 \alpha'(0)\sum_i\lt(\rmf{id}\otimes_\ratint \mu_{\phi^i}\rt)\rt|_{V_0}=\sum_i\lt.\lt(\rmf{id}\otimes_\ratint \mu_{\phi^i}\rt)\circ \ywrk{L}\rt|_{V_0}.\]
\qed
\end{proof}

\subsection{Heegner points and $p$-adic integrals}

As in Section \ref{p-adic uniformization of Shimura curves}, let $B$ be a definite quaternion algebra over $F$, which is ramified at all archimedean places and the finite places associated to prime  ideals dividing $\mathfrak{n}^-$.  Let $\doi{a}$ be a non-zero ideal of $\intring{F_\spp}$ relatively prime to $\doi{n}^-$ and let $R$ an Eichler  order of $B$ of level $\doi{a}$.
\begin{defn}
Let $K$ be a quadratic extension over $F$. An {\em optimal embedding of level $\doi{a}$} from K into B is a pair $(\Psi, b)\in\rmf{Hom}_{F-\rmf{alg}}(K,B)\times\lt(\yama{B}^\times\big/\yama{R}^\times\rt)$ satisfying
\[\Psi(\intring{K})=b\yama{R}b^{-1}\cap\Psi(K).\]
\end{defn}
For an optimal embedding $(\Psi, b)$ of level $\doi{a}$, we define
\[R_b\defeq B\cap b\yama{R}b^{-1}.\]
Then $R_b$ is an Eichler order of level $\doi{a}$, and $\Psi$ gives an embedding of $\intring{K}$ into $R_b$.  We define an action of $g\in \yama{B}^\times$ on an optimal embedding $(\Psi, b)$ by conjugation:
\[g\cdot(\Psi, b)\defeq(g\Psi g^{-1}, gb).\]
We denote by $[\Psi,b]$ the orbit of $(\Psi,b)$ associated with the subgroup $B^\times \subset\yama{B}^\times$.  The set of the orbits of optimal embeddings of level $\doi{a}$ is denoted by $\rmf{Emb}_F(K,B,\doi{a})$.  We also define an action of the ideal class group $\rmf{Pic}(\intring{K})$ of $K$  by identifying it with $\yama{K}^\times\big/K^\times\yama{\intring{K}^\times}$ and for any $\rho\in\rmf{Pic}(\intring{K})$ we often denote $\rho\cdot[\Psi,b]$ by $[\Psi^\rho, b^\rho]$.

The set $\rmf{Emb}_F(K,B,\doi{a})$ is described as follows. Let $\Sigma=\Sigma_0(\doi{a}, \doi{n}^-)$.  As in Section $\ref{another description}$, there is a bijection
\[
\xymatrix{
\rmf{Nrd}_{B/F}: B^\times\bbs\yama{B}^\times\big/B_\spp^\times\Sigma \ar[r]^(0.49){1:1} & F_+^\times\bbs\adele_{F, f}^\times\big/\yama{\intring{F}}^\times =:\nnm{Cl}_{F,+} 
},
\]
Then we have a decomposition
\begin{eqnarray} 
\label{Bdblcoset2}
\yama{B}^\times = \bigsqcup_{i=1}^h B^\times x_i B_\spp^\times\Sigma 
\end{eqnarray}
where $h=\#\nnm{Cl}_F^+$ is the narrow class number of $F$, the element $x_i \in \yama{B}^\times$ satisfies $(x_i)_\doi{p}=1$ and the images of $x_1,\dots,x_h$ by $\rmf{Nrd}_{B/F}$ gives a set of complete representatives of the finite group  $\nnm{Cl}_F^+$.  Via the bijection $\xi_\doi{a}$ as in Section $\ref{another description}$, we have a natural inclusion:
\begin{eqnarray}
\label{inclusiondayo}
\rmf{Emb}_F(K,B,\doi{a}) \hookrightarrow \bigsqcup_{i=1}^h\nami{\Gamma}_i\bbs\bigg(\rmf{Hom}_{\intring{F}-\rmf{alg}}\big(\intring{K}[1/p], R_i[1/p]\big)\times \ywrk{P}(\doi{a})\bigg).
\end{eqnarray}
We always regard the left hand side as a subset of right hand side.  For an element $[\Psi,b]\in\rmf{Emb}_F(K,B,\doi{a})$, if $[\Psi,b]$ belongs to the $i$-th component, we denote it by $[\Psi_i, L_i]$, which is the class of $(\Psi_i,L_i)\in \rmf{Hom}_{\intring{F}-\rmf{alg}}\big(\intring{K}[1/p], R_i[1/p]\big)$.

From now on, we suppose the following three conditions are satisfied:
\begin{ass}
\label{Heegner condition}
\begin{enumerate}
\item \label{he1}All prime ideals $\doi{q}$ of $\intring{K}$ dividing $\doi{a}$ split in $K$,
\item \label{he2}All prime ideals $\doi{q}$ of $\intring{K}$ dividing $\doi{n}^-$ are inert in $K$,
\item \label{he3}The prime ideal $\doi{p}$ dose not split in $K$.
\end{enumerate}
\end{ass}
\begin{rmk}
The above assumptions \ref{he1} and \ref{he2} are called the {\em Heegner condition}. Under these two assumptions, an optimal embedding of level $\doi{a}$ exists. 
(see \cite{Vig}, Theorem 3.1 and Theorem 3.2).
\end{rmk}

Fix an isomorphism $\iota_p:B_\doi{p}^\times\overset{\cong}{\rightarrow}\rmf{GL}_2(F_\doi{p})$.  For an optimal embedding $(\Psi, b)$, let $\tau_\Psi$ be the fixed point of $\iota_p\lt(\Psi(K_\doi{p})\rt)$ in $\hfpl{\uesen{\ratnum}_p}$ satisfying:
\[\iota_p\lt(\Psi(\alpha)\rt)\cdot\lt(\begin{array}{c}\tau_\psi \\ 1 \end{array}\rt)=\alpha\lt(\begin{array}{c}\tau_\psi \\ 1 \end{array}\rt), \]
for any $\alpha\in K_\doi{p}$.

Let $\psi_K:\rmf{Pic}(\intring{K})\rightarrow\lt\{\pm 1\rt\}$ be an unramified quadratic character over $K$. Via the isomorphism (\ref{Cerednik-Drinfeld theorem}), let \[(\xi_i)_{i=1}^h\in\bigoplus_{i=1}^h\rmf{Div}\lt(\Gamma_i\bbs\hfpl{\cpxnum_p}\rt)\otimes_\ratint\ratnum\] be the element corresponding to the Hodge class $\xi \in \rmf{Pic}(X_\ywrk{B})(F)\otimes_\ratint\ratnum$, which has degree one on each geometric component, and satisfies the relation \[T(\doi{l})\xi=(\norm{F/\ratnum}{\doi{l}}+1)\xi\] for any prime ideal $\doi{l}$ of $\intring{F}$ prime to $\doi{p}\doi{n}^+\doi{n}^-$.
\begin{defn}
\label{Heegner point}
For an unramified quadratic character $\psi_K:\rmf{Pic}(\intring{K})\rightarrow\{\pm1\}$, we define a point $\bfb{P}_{\psi_K}$ belonging to $\bigoplus_{i=1}^hJ\lt(X_i\rt)\otimes_\ratint\ratnum$ as follows:
\begin{eqnarray*}
\bfb{P}_{\psi_K}&\defeq&\ds\sum_{\rho\in\rmf{Pic}(\intring{K})}\psi_K(\rho)([\tau_{\Psi^\rho}]_{i_{\Psi^\rho}}-\xi_{i_{\Psi^\rho}})
\end{eqnarray*}
\end{defn}
\begin{rmk}
\label{Heegner point 2}
The image of $\bfb{P}_{\psi_K}$ by the projection to the $\Phi$-part can be described as \[\lt(\norm{F/\ratnum}{\doi{q}}+1-\alpha(\doi{q},\Phi)\rt)^{-1}\lt(\norm{F/\ratnum}{\doi{q}}+1-T(\doi{q})\rt)\ds\sum_{\rho\in\rmf{Pic}(\intring{K})}\psi_K(\rho)[\tau_{\Psi^\rho}]_{i_{\Psi^\rho}}\]
where $\doi{q}$ is a prime ideal of $\intring{F}$ prime to $\doi{p}\doi{n}^+\doi{n}^-$ satisfying $\norm{F/\ratnum}{\doi{q}}+1\neq\alpha(\doi{q},\Phi)$.
\end{rmk}
\begin{rmk}
\label{global point}
When $\doi{p}$ is inert in $K$,  by the complex multiplication theory, the point $\bfb{P}_{\psi_K}$ is global point and belongs to Jacobian $\bigoplus_{i=1}^hJ\lt(X_i\rt)(K_{\psi_K})\otimes_\ratint\ratnum$, where $K_{\psi_K}$ is the quadratic field over $K$ cut out by $\psi_K$. For details, see \cite{Mo1}, Section 4.2 and 4.3.
\end{rmk}

\begin{defn}
Denote $K=F(\lambda)$ where $\lambda\in K$ is the element such that $\lambda^2\in F$ and is totally negative.  The {\em partial $p$-adic $L$-function} $\ywrk{L}_p$ associated to $\Phi$ and a class optimal embedding $[\Psi,b]=[\Psi_i, L_i]$ of level $\doi{n}^+$ is defined by
\[\ywrk{L}_p(\Phi,\Psi, s)\defeq\begin{cases}\lt\langle C_i\cdot\lt|\rmf{Nrd}_{B/F}(x_i)\rt|_{\adele_{F,f}}\rt\rangle^{s/2}\theta_{\phi^i}^{\tau_\Psi}(s; L_i)& \mbox{if $\doi{p}$ is inert in $K$}\\ \lt\langle C_i\cdot\lt|\rmf{Nrd}_{B/F}(x_i)\rt|_{\adele_{F,f}}\rt\rangle^{s/2}\frac{1}{2}\lt(\theta_{\phi^i}^{\tau_\Psi}(s;L_i)+\theta_{\phi^i}^{\tau_\Psi}\big(s;\iota_\doi{p}\big(\Psi_i(\lambda)\big)L_i\big)\rt) & \mbox{if $\doi{p}$ is ramified in $K$},\end{cases}\]
where $\ds C_i={c}\big/\sqrt{\lt|N_{F/\ratnum}\lt(-\rmf{Nrd}_{B/F}(\lambda)\rt)\rt|_\bfw{R}}$.

We also define, for any unramified quadratic character $\psi_K:\rmf{Pic}(\intring{K})\rightarrow\{\pm1\}$,
\[\ywrk{L}_p(\Phi, \psi_K, s)\defeq\sum_{\rho\in\rmf{Pic}(\intring{K})}\psi_K(\rho)\ywrk{L}_p( \Phi,\Psi^\rho, s).\]
\end{defn}

\begin{prop}
\label{main results 2}
The following equalities hold:
\begin{eqnarray*}\lt.\frac{d}{ds}\ywrk{L}_p(\Phi,\psi_K,s)\rt|_{s=0}&=&I_\Phi(\nami{\bfb{P}}_{\psi_K})\\&=&\sum_{i=1}^h\logNorm_p^{X_i}(\bfb{P}_{{\psi_K},i})(\mu_{\phi^i})
\end{eqnarray*}
where $\nami{\bfb{P}}_{\psi_K}$ is a lift of $\bfb{P}_{\psi_K}$ in $\ds\bigoplus_{i=1}^h\rmf{Div}_0\lt(\Gamma_i\bbs\hfpl{\uesen{\ratnum}_p}\rt)\otimes_\ratint\ratnum$, and  $\bfb{P}_{{\psi_K},i}$ is the image of the projection to $J(X_i)$.
\end{prop}
\begin{proof}
The first equality follows from the definition of the indefinite integrals and Proposition \ref{Hecke actions of indefinite integrals} by using Remark \ref{Heegner point 2}.  The second equality follows from Theorem \ref{main results}.
\qed
\end{proof}

\subsection{$p$-adic $L$-functions of Hilbert modular forms}

Let $A$ be a modular abelian variety of $\rmf{GL}(2)-type$ over $F$ which is multiplicative at $\doi{p}$ (in addition, if $[F:\ratnum]$ is odd, suppose that $A$ is multiplicative at some prime other than $\doi{p}$).  Let $f$ be the Hilbert modular eigenform corresponding to $A$ of weight 2, level $\doi{n}\defeq \doi{p}\doi{n}^+\doi{n}^-$ and trivial character (here, $\doi{n}^+$ and $\doi{n}^-$ are non-zero ideals prime to $\doi{p}$, and $\doi{n}^-$ is a square free ideal such that $\#\{\doi{q}\text{: prime dividing }\doi{n}^-\}\equiv [F:\ratnum]\mod 2$).  Let $\bfb{f}_\infty$ be the Hida family of $f$ and we denote by $f_k$, $f_k^{\#}$ the weight $k$ specialization of $\bfb{f}_\infty$ and its $p$-stabilization respectively. Let $\Phi$ be a quaternionic automorphic form of weight 2 and level $\doi{p}\doi{n}^+$ corresponding to $f$ by the Jacquet-Langlands correspondence.  

 We briefly recall the $p$-adic $L$-functions and the two-variable $p$-adic $L$-functions of Hilbert modular forms (\cite{Dab}, \cite{Mo1}).  We shall explain them in a general situation.  Let $g$ be a cuspidal Hilbert modular eigenform of parallel weight $k$ and ordinary at  $\doi{p}$, and let $\chi$ (resp. $\psi$) be a finite order Hecke character of $F$ unramified outside $\doi{p}$ and infinite places (resp. at the conductor of $g$).  There exists {\em $p$-adic $L$-function} $L_p(s,g, \chi\psi)$ defined on $s\in \ratint_p$ such that
\[L_p(r,g,\chi\psi)=\left(1-\frac{\chi\psi\omega_F^{1-r}(\doi{p})\ywrk{N}(\doi{p})^{r-1}}{\alpha_1(\doi{p},g)}\right)\left(1-\frac{\lt(\chi\psi\omega_F^{1-r}\rt)^{-1}(\doi{p})\alpha_2(\doi{p},g)}{\ywrk{N}(\doi{p})^r}\right)\cdot\frac{L\lt(r,g,(\chi\psi\omega_F^{1-r})^{-1}\rt)}{\Omega_{g, \chi\psi, r}}\]
for any $r=1,\dots,k-1$. Here, $L(s,g,\chi\psi\omega_F^{1-r})$ is the complex $L$-function of $g$, $\alpha_1(\doi{p},g)$ (resp. $\alpha_2(\doi{p},g)$) is the $p$-adic unit root (resp. $p$-adic non-unit root) of the Hecke polynomial\[X^2-a(\doi{p},g)X+\epsilon_\doi{p}\ywrk{N}(\doi{p})^{k-1}\] with \[\epsilon_\doi{p}=\begin{cases}0 & \doi{p}|\text{(the conductor of g)} \\ 1 & \text{otherwise}, \end{cases}\] and  $\Omega_{g, \chi\psi, r}\in\cpxnum^\times$ is a complex number such that \[\frac{L(m,g,\lt(\chi\psi\omega_F^{1-r}\rt)^{-1})}{\Omega_{g, \chi\psi, r}}\in\uesen{\ratnum}.\]
We always regard an element in $\uesen{\ratnum}$ as in $\cpxnum_p$ through the fixed embedding $\uesen{\ratnum}\hookrightarrow\cpxnum_p$. Note that $\Omega_{g, \chi\psi,r}$ can be described as the product of {\em period} of $f$ and some non-zero constants (see \cite{Mo1}, Section 5.1). 

\subsection{Proof of the main theorems}

Let $L_p(s,f,\psi)$ be the $p$-adic $L$-function of $f$ with $\chi=1$.  Clearly, when $\psi(\doi{p})=\alpha_\spp$, we have trivial equality 
\[L_p(1,f,\psi)=0,\]
and we call it {\em exceptional zero}.

According to \cite{Mo2}, Theorem 6.8, we can construct the {\em two-variable $p$-adic $L$-function} $L_p(s,k, \psi)$ such that
\[L_p(s,m,\psi)=L_p\lt(s,f_m^{\#},\psi\rt) \]
for any $m\in U\cap\ratint$ and $m\equiv 2 \mod 2(p-1)$, where $U\in \ratint_p$ is an open neighborhood of $2$.

We prove the main theorem of this paper which is a generalization of the main results of \cite{BD1} and \cite{Mo1}.

\begin{thm}
\label{main results 3}
Let $\psi$ be a quadratic Hecke character of F of conductor prime to $\doi{n}$. We assume that the following conditions are satisfied:
\begin{eqnarray*}
\psi(\doi{p})&=&\alpha_\doi{p}, \\
\epsilon(f,\psi)&=&-1,
\end{eqnarray*}
where $\alpha_\doi{p}$ is the Hecke eigenvalue of $T(\doi{p})$ and $\epsilon(f,\psi)$ is the sign of the functional equation of the complex $L$-function $L(s,f,\psi)$.  Let $A$ be the abelian variety of $\rmf{GL}(2)$-type associated to $f$ as in beginning of the previous section. Then we have
\renewcommand{\labelenumi}{(\arabic{enumi})}
\begin{enumerate}
\item There exists a global point $P_\psi\in A(F^\psi)\otimes_\ratint\ratnum$ and $l\in \ratnum(f)^\times$ such that \[\lt.\frac{d^2}{dk^2}L_p(k/2,k,\psi)\rt|_{k=2}=l\cdot\lt(\logNorm_p^A(P_\psi)\rt)^2,\]
where $F^\phi$ is a quadratic extension corresponding to $\phi$.\label{thm2}
\item The element $P_\psi$ is of infinite order if and only if  the derivative of the complex $L$-function is nonzero $L'(1,A/F,\psi)\neq 0$, where $L(s,A/F,\psi)$ is the $L$-function of $A$. In that case, \[\rmf{dim}_{\ratnum(f)}\lt(A(F^\psi)\otimes_\ratint\ratnum \rt)_\psi=1.\] \label{thm3}
\end{enumerate}
\end{thm}
\begin{proof}
We can choose a quadratic Hecke character ${\psi'}$ which is unramified at the primes dividing $\doi{n}$ and the conductor is prime to that of $\psi$ such that 
\begin{eqnarray}L_p(1,2,\psi')\in \ratnum(f)^\times \label{tousiki C}\nonumber
\end{eqnarray}
and the quadratic extension over $F$ associated with $\psi\psi'$ is a CM extension with $\doi{p}$ inert.
  
By \cite{Mo1}, Proposition 5.1, 5.3, there exists an open neighborhood $U\in\ratint_p$ and a $p$-adic analytic function $\eta$ on $U$ such that
\begin{eqnarray}\ywrk{L}_p(\Phi, \psi_K, s)^2&=&\eta(s+2)L_p((s+2)/2,s+2,\psi)L_p((s+2)/2,s+2,{\psi'}),\label{tousiki A}\\
\eta(2)&\in& \ratnum(f)^\times,\label{tousiki B}\nonumber\end{eqnarray}
where $\psi_K$ is the Hecke character of $K$ associated with the composition of the fields $F_\psi$ and $F_{\psi'}$, which are the quadratic extensions over $F$ associated with $\psi$ and ${\psi'}$.

Since $L_p(k/2,k,\psi)$ has trivial zero and $\epsilon(f,\psi)=-1$, we see from the functional equation that the order of vanishing of $L_p(k/2,k,\psi)$ is at least 2:
\[L_p(1,2,\psi)=L_p'(1,2\psi)=0.\]
Thus by calculating the second derivative at $s=0$ of (\ref{tousiki A}), we have
\[2\lt(\lt.\frac{d}{ds}\ywrk{L}_p(\Phi,\psi_K,s)\rt|_{s=0}\rt)^2=\eta(2)L_p(1,2,\psi')\lt.\frac{d^2}{dk^2}L_p(k/2,k,\psi)\rt|_{k=2}.\]
Put $l:=2\eta(2)^{-1}L_p(1,2,\psi')^{-1}\in \ratnum(f)^\times$, we have
\[\lt.\frac{d^2}{dk^2}L_p(k/2,k,\psi)\rt|_{k=2}=l\lt(\lt.\frac{d}{ds}\ywrk{L}_p(\Phi,\psi_K,s)\rt|_{s=0}\rt)^2.\]
On the other hand, by Proposition \ref{main results 2}, the right hand side is equal to \[l\lt(\sum_{i=1}^h\logNorm_p^{X_i}(\bfb{P}_{{\psi_K},i})(\mu_{\phi^i})\rt)^2,\]
where $\bfb{P}_{\psi_K}$ is a global point and belongs to the Jacobian $\bigoplus_{i=1}^hJ\lt(X_i\rt)(K_{\psi_K})\otimes_\ratint\ratnum$, where $\coprod_i X_i$ is the Shimura curve as in Section \ref{p-adic uniformization of Shimura curves} (see Remark \ref{global point}).

By \cite{Z1} and the proof of \cite{Mo1}, Corollary 4.2, we have \[\uesen{\bfb{P}}_{\psi_K}\in A(F^\phi)\otimes_\ratint\ratnum,\]
where $\uesen{\bfb{P}}_{\psi_K}$ is the image of $\bfb{P}_{\psi_K}$ by the projection from $\bigoplus_{i=1}^hJ\lt(X_i\rt)(K_{\psi_K})$ to $A(K_{\psi_K})$.  We have (\ref{thm2}) by putting $P_\psi\defeq\uesen{\bfb{P}}_{\psi_K}$.

For (\ref{thm3}), by the results of Zhang ($\cite{Z2}$), we have
\[P_\psi \text{ is of infinite order} \Leftrightarrow L'(1,f/K, \psi_K)\neq0\]
where $L(s, f/K, \psi_K)$ is the Rankin-Selberg convolution $L$-function of $f$ and $\psi_K$.  On the other hand, we have
\[L(s, f/K, \psi_K)=L(s, f, \psi)\cdot L(s,f,\psi').\]
Thus, by the choice of $\psi$ and $\psi'$, we have
\[L'(1,f/K,\psi_K)\neq0 \Leftrightarrow L'(1,f,\psi)\neq0 \Leftrightarrow L'(1,A/F,\psi)\neq 0.\]
The second equivalence follows from $\cite{Z1}$, Theorem A.  Therefore, we have the first assertion of (\ref{thm3}).  The second assertion of (\ref{thm3}) is obtained by the theorem of Kolyvagin-Logachev \cite{KL}.
\qed
\end{proof}

\begin{table}[h]
\begin{center}
\begin{tabular}{l}
Isao Ishikawa \\
Department of Mathematics\\
Faculty of Science\\
Kyoto University\\
Kyoto 606-8502\\
Japan\\
iishikawa@math.kyoto-u.ac.jp

\end{tabular}
\end{center}

\end{table}

\end{document}